\documentclass[reqno]{amsart}

\usepackage{amsmath, mathtools}
\usepackage{amssymb}
\usepackage{bbm}
\usepackage[margin=1in]{geometry}
\usepackage{xcolor}
\usepackage{comment}
\usepackage{graphicx}
\usepackage{empheq}

\theoremstyle{plain}
\newtheorem{thm}{Theorem}[section]
\newtheorem{lem}{Lemma}[section]

\theoremstyle{definition}

\theoremstyle{remark}

\newtheorem{rmk}{Remark}[section]

\newcommand{\norm}[1]{\left\lVert#1\right\rVert}

\newcommand{\osc}{\operatorname{osc}}

\newcommand{\dv}{\operatorname{div}}

\newcommand{\M}{\mathbb{M}}
\newcommand{\om}{\omega}
\newcommand{\Om}{\Omega}
\newcommand{\sig}{\sigma}
\newcommand{\eps}{\epsilon}
\newcommand{\Ups}{\Upsilon}
\newcommand{\R}{\mathbb{R}}
\newcommand{\1}{\mathbbm{1}}
\newcommand{\A}{\mathcal{A}}
\newcommand{\N}{\mathcal{N}}
\newcommand{\BMO}{\operatorname{BMO}}
\newcommand{\loc}{\operatorname{loc}}
\newcommand{\Id}{\operatorname{Id}}
\newcommand{\Log}[1]{(#1)^{\log}}
\newcommand{\ov}[1]{\overline{#1}}
\newcommand{\til}[1]{\tilde{#1}}
\newcommand{\ovM}{\overline{\M}}
\newcommand{\ovom}{\overline{\omega}}
\newcommand{\data}{\textbf{data}}
\newcommand{\inner}[2]{\left\langle #1, #2 \right\rangle}

\def\Xint#1{\mathchoice
   {\XXint\displaystyle\textstyle{#1}}%
   {\XXint\textstyle\scriptstyle{#1}}%
   {\XXint\scriptstyle\scriptscriptstyle{#1}}%
   {\XXint\scriptscriptstyle\scriptscriptstyle{#1}}%
   \!\int}
\def\XXint#1#2#3{{\setbox0=\hbox{$#1{#2#3}{\int}$}
     \vcenter{\hbox{$#2#3$}}\kern-.5\wd0}}
\def\dashint{\Xint-}

\usepackage{cite}
\usepackage{hyperref}
\hypersetup{ colorlinks = true, urlcolor = blue, linkcolor = blue, citecolor = red }
\usepackage{enumitem}
\usepackage{nameref}
\makeatletter
\let\orgdescriptionlabel\descriptionlabel
\renewcommand*{\descriptionlabel}[1]{%
  \let\orglabel\label
  \let\label\@gobble
  \phantomsection
  \edef\@currentlabel{#1\unskip}%
  \let\label\orglabel
  \orgdescriptionlabel{#1}%
}
\numberwithin{equation}{section}

\begin{document}
	
\title[Generalized Double Phase Equations with Matrix Weights]{ Calder\'on--Zygmund Estimates for Generalized Double Phase Equations with Matrix Weights}

\author[Byun]{Sun-Sig Byun}
\address{Department of Mathematical Sciences and Research Institute of Mathematics,
	Seoul National University, Seoul 08826, Republic of Korea}
\email{byun@snu.ac.kr}

\author[Kim]{Hongsoo Kim}
\address{Department of Mathematical Sciences, Seoul National University, Seoul 08826, Republic of Korea}
\email{rlaghdtn98@snu.ac.kr}

\thanks {S.-S. Byun was supported by Mid-Career Bridging
Program through Seoul National University. H. Kim was supported by the National Research Foundation of Korea(NRF) grant funded by the Korea government [Grant No. 2022R1A2C1009312].}

\makeatletter
\@namedef{subjclassname@2020}{\textup{2020} Mathematics Subject Classification}
\makeatother

\everymath{\displaystyle}
\bibliographystyle{amsplain}

\begin{abstract}
	We prove Calder\'on--Zygmund estimates for generalized double phase equations with Orlicz growth and variable matrix weights.
    The operator combines a non-uniformly elliptic double phase structure with a degenerate or singular matrix weight satisfying a small log-$\mathrm{BMO}$ condition.
    Under appropriate structural assumptions, we show that higher integrability of the weighted datum yields higher integrability of the weighted gradient of weak solutions.
    Our results extend the existing Calderón–Zygmund theory for double phase problems and weighted elliptic equations to a unified framework capturing the interaction between Orlicz growth and matrix-weighted structures, thereby building upon and unifying the results in \cite{Baasandorj20} and \cite{Cho25}.
\end{abstract}

\maketitle

\section{Introduction} \label{sec1}
In this paper, we study the following double-phase problem with Orlicz growth and a matrix-valued weight, whose prototype is given by
\begin{align*}
    &\dv\left(\M G'(|\M Du|)\frac{\M Du}{|\M Du|}+a(x)\M H'(|\M Du|)\frac{\M Du}{|\M Du|}\right) \\ &\quad = \dv\left(\M G'(|\M F|)\frac{\M F}{|\M F|}+a(x)\M H'(|\M F|)\frac{\M F}{|\M F|}\right) \text{ in }\Omega.
\end{align*}
where $\Om \subset \R^n$ is a bounded domain.
Here $G,H\in \N$ (see Section \ref{sec3.2}) and coefficient $a : \Omega \rightarrow \mathbb{R}$ satisfies
\begin{align} \label{GHcond}
     \sup_{t\geq0} \frac{H(t)}{G(t)+G^{\min\{1+\frac{\alpha}{n},1+i_0\}}(t)}<\infty, \quad 0\leq a(x) \in C^{0,\alpha}(\Omega)
\end{align}
for some $i_0 < i(H)$.
The matrix-valued weight $\M:\Om \rightarrow\R^{n\times n}$ is assumed to be symmetric and positive definite almost everywhere and satisfies
\begin{align} \label{Mcond}
    |\M(x)||\M^{-1}(x)| \leq \Lambda.
\end{align}
for almost every $x \in \Om$, where $|\cdot|$ denotes the spectral norm. 
Then \eqref{Mcond} implies that for any $\xi \in \R^n$ and almost every $x \in \Om$, the following comparability holds.
\begin{align} \label{Mwcomp}
    \Lambda^{-1} |\M(x)| | \xi| \leq |\M(x)\xi| \leq |\M(x)||\xi|.
\end{align}

Our objective is the Calderon-Zygmund estimates for this generalized double-phase problem with a matrix variable weight.
Specifically, we prove that under sufficient conditions on the weight $\M$, the following implication holds for any Young function $\Ups \in \N$,
\begin{align} \label{goal}
    G(|\M F|)+a(x)H(|\M F|) \in L^\Upsilon_{\loc} \Longrightarrow G(|\M Du|)+a(x)H(|\M Du|) \in L^\Upsilon_{\loc}
\end{align}

When $\M$ is the identity matrix $\Id$,  and $G(t)=t^p$, $H(t)=t^q$, the equation reduces to the classical double phase problem introduced by Zhikov \cite{Zhikov86,Zhikov95}. 
The regularity theory for such non-uniformly elliptic problems was developed in 
\cite{Colombo15,Colombo152,Baroni18,Colombo16,Filippis19}
under the optimal gap condition:
\begin{align} \label{pqcond}
    \frac{q}{p} \leq 1+\frac{\alpha}{n}.
\end{align}
This framework was later extended to general Orlicz growth in 
\cite{Byun20,Baasandorj20,Baasandorj25},
where the corresponding structural condition becomes
\begin{align}\label{wantGHcond}
    \sup_{t\geq0} \frac{H(t)}{G(t)+G^{1+\frac{\alpha}{n}}(t)}<\infty,
\end{align}
which reduces to \eqref{pqcond} in the case of power-type growth.
Note that if $i(H)>\frac{\alpha}{n}$, then the condition above coincides with \eqref{GHcond}.
In particular, \eqref{goal} was proved in \cite{Baasandorj20} for the case $\M = \Id$.

The purpose of the present paper is to extend this Orlicz double phase Calder\'on--Zygmund theory from constant matrix weight $\M=\Id$ to variable matrix weights under suitable smallness assumptions on the oscillation of $\M$.
In this setting, the presence of the matrix weight $\M$ interacts nontrivially with the generalized growth structure \eqref{GHcond}.
As a result, the problem requires a combined treatment of the double phase framework with Orlicz growth and matrix-weight Calderón–Zygmund theory, and does not follow directly from either framework alone.

While the preceding literature focuses on the challenges of double-phase growth, a separate but equally vital line of research addresses regularity in the presence of matrix-valued weights $\M$. 
Calder\'on--Zygmund estimates for elliptic equations with matrix weights were developed in \cite{Balci22}, where for the $p$-growth problem
\begin{align} \label{pprob}
    \dv( \M A(x,\M Du)) = \dv( \M A(x,\M F))
\end{align}
with $A(x,z)=|z|^{p-2}z$ and $\M$ satisfying \eqref{Mcond}, it was shown that for any $\gamma>1$,
\begin{align*}
    |\M F|^p \in L^\gamma_{\loc} \Longrightarrow |\M Du|^p \in L^\gamma_{\loc}, 
\end{align*}
under a small log-$\BMO$ condition on $\M$.
Further developments in this direction appear in \cite{Balci23,Yang25,Cho242,Cho24}.
These studies identify the small log-$\BMO$ assumption as the correct quantitative regime for matrix-weight Calder\'on--Zygmund theory.

The first integration of matrix weights and double phase structure was achieved in \cite{Cho25}, where the power-type operator was treated under \eqref{pqcond} together with a log-$\BMO$ condition on $\M$. 
The present work can be viewed as a natural extension of \cite{Cho25}, advancing the theory from the power-type double phase with matrix weights to the fully generalized Orlicz double phase setting. 

A crucial ingredient enabling this extension is the absence of the Lavrentiev phenomenon for the weighted generalized double phase functional,
\begin{align*}
    \mathcal{P}(v,\Omega) := \int_{\Omega}G(|\M Dv|)+a(x)H(|\M Dv|) \ dx,
\end{align*}
proved in \cite{Byun25} under the assumption \eqref{GHcond}.
This density result allows approximation by smooth functions in the weighted Musielak--Orlicz setting and is essential for the comparison arguments underlying \eqref{goal}. 
We emphasize that the stronger condition \eqref{GHcond}, compared with \eqref{wantGHcond}, enters only through this density property; if the absence of the Lavrentiev phenomenon were available under \eqref{wantGHcond}, then the Calder\'on--Zygmund estimate \eqref{goal} would remain valid under that optimal condition as well.

The paper is organized as follows.
In Section \ref{sec2}, we state the assumptions and the main result, Theorem \ref{Main} and introduce notation and preliminaries about orlicz functions, weights and weighted function space.
In Section \ref{sec3}, we present the absence of Lavrentiev phenomenon and establishes weighted Sobolev-Poincare inequalities.
In Section \ref{sec4}, we obtain regularity properties of suitable reference problems.
Finally, Section \ref{sec5} contains the proof of Theorem \ref{Main}.

\section{Preliminaries and Main Results} \label{sec2}
In this section, we introduce the notation and preliminary results used throughout the paper, and state the main result under the corresponding structural assumptions.

\subsection{Notation}
We collect here the notation that will be used throughout the paper.
We write $a \approx b$ if $c^{-1}a \le b \le ca$ for some constant $c>0$.
We write $B_r(x_0) := \{x \in \R^n : |x-x_0| < r\}$ as a open ball with centered at $x_0$ with radius $r$.
For any $p>1$, we write $p' = \frac{p}{p-1}$.
For a locally integrable function $f$ on $\R^n$, we denote the average of $f$ over a ball $B$ by
\begin{align*}
     (f)_B= \dashint_B f dx = \frac{1}{|B|} \int_B f dx.
\end{align*}
\subsection{Young function} \label{sec3.2}
In this subsection, we recall several basic facts about Young function that will be used throughout the paper.
We write $\Phi \in \N$ if $\Phi \in C^1([0,\infty))\cap C^2((0,\infty))$ is convex and increasing, satisfies $\Phi(0)=0$, $\lim_{t\rightarrow \infty} \Phi(t)= \infty$, and there exist constants $0<i(\Phi)\leq s(\Phi)$ such that,
for any $t>0$,
\begin{align} \label{Ncond}
    i(\Phi) \leq \frac{t\Phi''(t)}{\Phi'(t)}  \leq s(\Phi).
\end{align}
For a Young function $\Phi \in \N$, it follows that
\begin{align*}
    \Phi(t) \approx t\Phi'(t) \approx t^2\Phi''(t),
\end{align*}
uniformly for $t>0$ with implicit constants depending only on $i(\Phi)$ and $s(\Phi)$.
Moreover, for any $\lambda>0$,
\begin{align} \label{lamt}
    \min \{ \lambda ^{1+i(\Phi)},\lambda ^{1+s(\Phi)}\} \Phi(t) \leq \Phi(\lambda t) \leq \max \{ \lambda ^{1+i(\Phi)},\lambda ^{1+s(\Phi)}\} \Phi(t).
\end{align}
We will also use the following form of Young's inequality:
\begin{align} \label{young}
    t \Phi'(s) + s\Phi'(t) \leq \epsilon \Phi(t) + \frac{c}{\epsilon^{s(\Phi)}}\Phi(s)
\end{align}
for any $s,t \geq 0$ and $\epsilon\in(0,1)$.

We also define the auxiliary vector field
\begin{align*}
    V_\Phi(x) :=\left[ \frac{\Phi'(|z|)}{|z|}\right]^{1/2}z.
\end{align*}
Then for a vector field $A$ satisfying \eqref{AGHcond}, we have the following inequality:
\begin{align} \label{VGH}
     |V_G(z_1) - V_G(z_2)|^2 +a(x)|V_H(z_1) - V_H(z_2)|^2\leq \left\langle A(x,z_1)-A(x,z_2),z_1-z_2 \right\rangle.
\end{align}
\subsection{Weights and the weighted Sobolev--Orlicz space}

We briefly recall some facts on Muckenhoupt weights. 
Let $\mu : \mathbb{R}^n \to [0,\infty)$ be a locally integrable function and let $1<p<\infty$. 
We say that $\mu$ is an $\A_p$-weight if $\mu \in L^1_{\loc}(\mathbb{R}^n)$ and
\begin{align*}
    [\mu]_{\A_p} := \sup_B \left( \dashint_B \mu \, dx \right) 
    \left( \dashint_B \mu^{-\frac{1}{p-1}} \, dx \right)^{p-1} < \infty,
\end{align*}
where the supremum is taken over all balls $B \subset \mathbb{R}^n$.

Let $\M : \Omega \to \mathbb{R}^{n\times n}$ be a symmetric positive definite matrix weight, and set $\om(x) := |\M(x)|$, where $|\cdot|$ denotes the spectral norm. 
We define the logarithmic averages of $\M$ and $\om$ over a ball $B$ by
\begin{align*}
    (\M)^{\log}_B := \exp\left( \dashint_B \log \M(x)\, dx \right), 
    \qquad 
    (\om)^{\log}_B := \exp\left( \dashint_B \log \om(x)\, dx \right),
\end{align*}
where $\exp$ and $\log$ denote the matrix exponential and matrix logarithm, respectively.

In view of the structural assumptions of $\M$ mentioned before, the matrix $(\M)^{\log}_B$ is symmetric positive definite and satisfies the comparability
\begin{align} \label{Mlog}
    \Lambda^{-1} |(\M)^{\log}_B|\, |\xi| 
    \le |(\M)^{\log}_B \xi| 
    \le |(\M)^{\log}_B|\, |\xi|
    \qquad \text{for all } \xi \in \mathbb{R}^n.
\end{align}

We recall that for a domain $\Omega \subset \mathbb{R}^n$, the space $\BMO(\Omega)$ consists of all functions $f \in L^1_{\loc}(\Omega)$ such that
\begin{align*}
    |f|_{\BMO(\Omega)} := \sup_B \dashint_B |f(x) - (f)_B|\, dx < \infty.
\end{align*}

Finally, we note that the scalar weight $\om$ inherits the logarithmic BMO control of $\M$ in the sense that
\begin{align*}
    |\log \om|_{\BMO(\Omega)} \le 2\, |\log \M|_{\BMO(\Omega)}.
\end{align*}

We stress that the small log-BMO condition on $\M$ is crucial in what follows. In particular, it provides quantitative control of the oscillation of the weight and allows comparison with its logarithmic averages, as made precise in the next lemma.
\begin{lem}[{\cite[Proposition 5]{Balci22}}] \label{M-ovM}
    There exist $\mu(n,\Lambda), c(n,\Lambda)>0$ such that if $|\log \M|_{\BMO(B)} \leq \frac{\mu}{s}$ for some $s \geq1$, then
    \begin{align*}
        \left( \dashint_B \left| \frac{\M -(\M)^{\log}_B}{(\M)^{\log}_B}\right|^s \right)^{1/s} \leq cs|\log \M|_{\BMO(B)}.
    \end{align*}
    The same holds with $\M$ replaced by $\om$.
\end{lem}
The next lemma provides additional integrability and balance properties of the associated scalar weight under the small log-BMO condition.
\begin{lem}[{\cite[Proposition 6]{Balci22}}]\label{oms}
    There exists $\beta(n,\Lambda)>0$ which satisfies the followings:
    \begin{enumerate}
        \item If $|\log \om|_{BMO(B)} \leq \frac{\beta}{s}$ for $s\geq1$, then
        \begin{align*}
            \left(\dashint_B \om^s \right)^{1/s} \leq 2(\om)^{\log}_B \quad \text{ and } \quad \left(\dashint_B \om^{-s} \right)^{1/s} \leq \frac{2}{(\om)^{\log}_B}.
        \end{align*}
        \item If $|\log \om|_{BMO(B)} \leq \beta \min \left\{ \frac{1}{p}, \frac{1}{(\theta p)'}\right\}$ for $p \in (1,\infty)$ and $\theta \in (0,1]$ with $\theta p >1$, then
        \begin{align*}
            \sup_{B' \subset B} \left( \dashint_{B'} \om^p \right)^{1/p} \left( \dashint_{B'} \om^{-(\theta p)'} \right)^{1/(\theta p)'} \leq 4,
        \end{align*}
        where $B'\subset B$ is any ball in $B$.
    \end{enumerate}
\end{lem}
For a Young function $\Phi \in \N$ and a weight $\om$, we define the weighted Orlicz space $L^\Phi_\om(\Omega)$ as the set of all measurable function $f$ on $\Omega$ satisfying
\begin{align*}
    \int_\Omega \Phi(|f(x) \om(x)|) \, dx < \infty.
\end{align*}
Then, in light of Lemma \ref{oms}, this space is a reflexive Banach space with the norm
\begin{align*}
    \norm{f}_{L^\Phi_\om(\Omega)}= \inf_{\lambda>0} \left\{\int_\Omega \Phi\left(\frac{|f(x) \om(x)|}{\lambda}\right) \, dx \leq 1\right\},
\end{align*}
provided that $|\log \om|_{BMO(\Omega)} \leq \delta$ with small enough $\delta>0$ depending on $n,\Lambda, i(\Phi)$, and $s(\Phi)$.
We also define weighted the Orlicz-Sobolev space $W^{1,\Phi}_\om(\Omega)$ as the set of all functions $f \in W^{1,1}(\Om)$ with $f, |Df| \in L^\Phi_\om(\Omega)$ with the norm $\norm{f}_{W^{1,\Phi}_\om(\Omega)} = \norm{f}_{L^\Phi_\om(\Omega)} + \norm{|Df|}_{L^\Phi_\om(\Omega)}$.

We introduce weighted Sobolev–Poincaré inequalities in Orlicz spaces.
\begin{lem}[{\cite[Lemma 2.5]{Yang25}}] \label{Orl-Sob}
    Let $\Phi \in \N$ and $v \in W^{1.\Phi}_\om (2B)$ for some ball $B=B_r$.
    Then there exist $\theta\in(0,1)$, $\delta >0$, and $c>0$, depending on $n,\Lambda , i(\Phi)$ and $s(\Phi)$ such that if $|\log \om|_{BMO(B)} \leq \delta$, then
    \begin{align*}
        \dashint_B \Phi \left( \frac{|v-(v)_B|}{r} \om \right) \, dx \leq c \left( \dashint_B \Phi^\theta(|Dv|\om) \, dx \right)^{1/\theta}.
    \end{align*}
\end{lem}
We state the following technical lemma, which will be used later.
\begin{lem}[{\cite[Lemma 3.4]{Han97}}]  \label{tech}
    Let $f:[R/2, R] \rightarrow [0,\infty)$ be a bounded function and $A,B \geq0$, $s,t\geq0$, $\theta \in (0,1)$.
    Assume that
    \begin{align*}
        f(r_1) \leq \theta f(r_2) + \frac{A}{(r_2-r_1)^s}  + \frac{B}{(r_2-r_1)^t}
    \end{align*}
    for any $R/2 \leq r_1< r_2 \leq R$.
    Then there there exists $c=c(s,t,\theta)>0$ such that
    \begin{align*}
        f(R/2) \leq c\left( \frac{A}{R^s} +  \frac{B}{R^t}\right).
    \end{align*}
\end{lem}

\subsection{Main results}
We study the distributional solution of the following weighted double-phase equation:
\begin{align} \label{mainpde}
    \dv\left(\M A(x,\M Du)\right) =\dv \M B(x,\M F) \quad \text{in } \Om.
\end{align}
Throughout this paper, the vector field $A : \Om \times\R^n \rightarrow\R^n$ and its derivative $D_zA: \Om \times\R^n\setminus\{0\} \rightarrow\R^n$ are assumed to be Carath\'eodory maps.
We impose the following structural assumptions: there exist constants $0<\nu\leq L<\infty$ such that
\begin{align} \label{AGHcond}
    \begin{cases}
        |A(x,z)| + |D_zA(x,z)||z|\leq L\left(\frac{G(|z|)}{|z|} + a(x)\frac{H(|z|)}{|z|} \right), \\
        \nu \left(\frac{G(|z|)}{|z|^2}+a(x)\frac{H(|z|)}{|z|^2}\right)|\xi|^2 \leq \left\langle D_z A_G(x,z)\xi,\xi\right\rangle, \\
        |A(x_1,z)-A(x_2,z)| \leq L|a(x_1)-a(x_2)|\frac{H(|z|)}{|z|},
    \end{cases}
\end{align}
for a.e. $x \in \Om$ and for all $z \in \R^n\setminus\{0\}, \xi\in\R^n$, where $G,H \in \N$.
Moreover, the vector field $B:\Om\times\R^n \rightarrow\R^n$ is assumed to be a Carath\'eodory map with the growth condition
\begin{align*}
    |B(x,z)| \leq L\left(\frac{G(|z|)}{|z|} + a(x) \frac{H(|z|)}{|z|}\right).
\end{align*}
We also assume that $0\leq a(x) \in C^{0,\alpha}(\Omega)$ and
\begin{align} \label{kGHcond}
    \sup_{t\geq0} \frac{H(t)}{G(t)+G^{1+i_0}(t)}<\infty, \quad \kappa:=\sup_{t\geq0} \frac{H(t)}{G(t)+G^{1+\frac{\alpha}{n}}(t)} <\infty.
\end{align}
for some $i_0 < i(H)$.

We set
\begin{align*}
    \Psi(x,z) :=G(|z|)+a(x)H(|z|).
\end{align*}
A function $u \in W^{1,1}(\Omega)$ with $\Psi(x, \M Du) \in L^1(\Omega)$ is called a distributional solution to \eqref{mainpde} if
\begin{align*}
    \int_\Omega \inner{A(x,\M Du)}{\M D\varphi} \, dx
    = \int_\Omega \inner{B(x,\M F)}{\M D\varphi} \, dx,
\end{align*}
for every $\varphi \in C^{\infty}_0(\Om)$.
For brevity, we write $\data$ to denote a collection of parameters depending only on the known data:
\begin{align*}
    \data = \data\left(n, \nu, L, \Lambda, \kappa, i(G),s(G),i(H),s(H), \alpha,
    \norm{a}_{C^{0,\alpha}(\Omega)}, \norm{\Psi(x, \M Du)}_{L^1(\Omega)} \right).
\end{align*}
Throughout the paper, $c>0$ denotes a universal constant depending only on $\data$, which may vary from line to line.

We are now ready to state the main result of the paper.
\begin{thm}  \label{Main}
    Let $u \in W^{1,1}(\Om)$ be a distributional solution to \eqref{mainpde} with $\Psi(x,\M Du) \in L^1(\Om)$ under the assumptions \eqref{Mcond}, \eqref{AGHcond} and \eqref{kGHcond}.
    Then for every $\Ups \in \mathcal{N}$, there exists a small number $\delta = \delta(\data, s(\Ups))>0$ such that if 
    \begin{align*}
        |\log \M|_{\BMO(\Omega)} \leq \delta,
    \end{align*}
    then the following implication,
    \begin{align*}
    \Psi(x,\M F) \in L^{\Upsilon}_{loc}(\Om) \Rightarrow \Psi(x,\M Du) \in L^{\Ups}_{loc}(\Omega),
    \end{align*}
    holds.
    Moreover, there exists a radius $r_0 = r_0 (\data, s(\Ups))>0$ such that
    \begin{align*}
        \dashint_{B_{R/2}}\Ups(\Psi(x,\M Du)) \, dx \leq c\Ups\left( \dashint_{B_R} \Psi(x,\M Du) \, dx \right) + c \dashint_{B_R} \Ups(\Psi(x,\M F)) \, dx,
    \end{align*}
    for every ball $B_R \Subset \Om$ with $R \leq r_0 \leq 1$ and for some $c=c(\data, s(\Ups))>1$.
\end{thm}

\section{Lavrentiev phenomenon and Sobolev-Poincar\'e inequality} \label{sec3}
We begin by recalling a density result which ensures the absence of the
Lavrentiev phenomenon for the weighted Orlicz double-phase functional proved in \cite{Byun25}.
This result allows approximation by smooth functions under a small
log-BMO assumption on the matrix weight.
\begin{lem}[\cite{Byun25}] \label{Lav}
Let ${B\Subset B'\Subset\Om}$, $G,H \in \N$ satisfy \eqref{GHcond} and $\M$ satisfy \eqref{Mcond}.
Then there exists $$\delta_0=\delta_0(i(G),s(G),i(H),s(H),\alpha,n,i_0)>0$$ such that the following holds:
if $f \in W^{1,1}(\Omega)$ with
$\Psi(x,\M Df)\in L^1(B')$
and
\begin{align*}
    |\log\M|_{\BMO(B')}\leq\delta_0,
\end{align*} then there
exists a sequence $\{f_{k}\}\subset C^{\infty}(B)$ such that 
\begin{align*}
    f_k \rightarrow f \text{ in } W^{1,1}(\Omega) \quad \text{ and } \quad \int\Psi(x,\M Df_k) \, dx  \rightarrow  \int\Psi(x,\M Df) \, dx.
\end{align*}
\end{lem}
This lemma guarantees that smooth functions are dense in the
weighted Sobolev–Orlicz space under the smallness condition on the
logarithmic oscillation of $\M$.
Using this approximation property, we can justify the use of
Sobolev-type test functions in the weak formulation.
\begin{lem} \label{testfun}
    Let $B \Subset B' \Subset \Omega$.
    Assume $G,H \in \N$ and $\M$ satisfy the hypotheses of Lemma \ref{Lav}.
    Let $f \in W^{1,1}(B)$ with $\Psi(x, \M Df) \in L^1(B)$ be a distributional solution of
    \begin{align*}
        \dv \M A(x, \M Du) = \dv \M B(x, \M F) \quad \text{ in } B,
    \end{align*}
    where $F:B\to\R^n$ satisfies $B(x,\M F)\in L^1(B)$.
    Then for every $\varphi \in W^{1,1}_0(B)$ with $\Psi(x,\M D\varphi)\in L^1(B)$, we have
    \begin{align*}
        \int_B \inner{A(x,\M Df)}{\M D\varphi} \, dx = \int_B \inner{B(x,\M F)}{\M D\varphi} \, dx.
    \end{align*}
\end{lem}
\begin{proof}
    By Lemma \ref{Lav}, there exists a sequence $\{\varphi_k\}\subset C^\infty_0(B)$ such that $D\varphi_k\to D\varphi$ a.e. and $\Psi(x,\M D\varphi_k)\to \Psi(x,\M D\varphi)$ in $L^1(B)$.
    Passing to the limit in the weak formulation for smooth test functions and using dominated convergence yields the desired identity.
\end{proof}
For the remainder of the paper, we assume
$$|\log\M|_{\BMO(\Omega)} \leq \delta_0,$$
so that every function $\varphi\in W^{1,1}_0(\Omega)$ satisfying
$\Psi(x,\M D\varphi)\in L^1(\Omega)$
is an admissible test function for distributional solutions.

The following lemma serves as a bridge that allows us to replace the variable weight $\om$ by its average $(\om)^{\log}_B$ at the expense of a exponent.
\begin{lem} \label{MovM}
    Let $0<\sig <\sig_*$. Then there exist $\delta>0$ and $c>0$ depending on $\data$, $\sigma$, and $\sigma_*$ such that if $|\log \om|_{\BMO} \leq \delta $, then for any $f \in L^1(B)$ satisfying $\Psi(x,\om f) \in L^{1+\sig_*}$, we have $\Psi(x, \overline{\om}f) \in L^{1+\sig}(B)$ where $\overline{\om} = (\om)^{\log}_B$ and
    \begin{align} \label{MovMest}
        \left(\dashint_B \Psi(x, \ov{\om}f)^{1+\sig} \, dx\right)^{\frac{1}{1+\sig}} \leq c\left(\dashint_B \Psi(x, \om f)^{1+\sig_*}\, dx\right)^{\frac{1}{1+\sig_*}}.
    \end{align}
\end{lem} 
\begin{proof}
Using \eqref{lamt} and H\"older inequality with exponent $\frac{1+\sig_*}{1+\sig}$ and its conjugate $\frac{1+\sig_*}{\sig_*-\sig}$, we have
\begin{align*}
    \dashint_B \Psi(x, \ov{\om}f)^{1+\sig} \, dx &\leq c \dashint_B G\left(\frac{|\ov{\om}|}{|\om|}|\om f|\right)^{1+\sig} + \left(a(x)H\left(\frac{|\ov{\om}|}{|\om|}|\om f|\right)\right)^{1+\sig} \, dx \\
    &\leq c \dashint_B  \zeta_G^{1+\sigma} G\left(|\om f|\right)^{1+\sig}+\zeta_H^{1+\sigma} \left(a(x)H\left(|\om f|\right)\right)^{1+\sig}\,dx \\
    &\leq c \left( \dashint_B \Psi(x, \om f)^{1+\sig_*} \,dx \right)^{\frac{1+\sig}{1+\sig_*}} \left( \dashint_B \zeta_G^{\frac{(1+\sig)(1+\sig_*)}{\sig_*-\sig}} +\zeta_H^{\frac{(1+\sig)(1+\sig_*)}{\sig_*-\sig}}\,dx \right)^{\frac{\sig_*-\sig}{1+\sig_*}},
\end{align*}
where for $\Phi \in \{G,H\}$, $\zeta_\Phi(x)$ is defined by
\begin{align*}
    \zeta_\Phi(x) = \left\{ \left(\frac{|\ov{\om}|}{|\om|} \right)^{1+i(\Phi)} +  \left(\frac{|\ov{\om}|}{|\om|} \right)^{1+s(\Phi)}\right\}.
\end{align*}
Choosing $\delta>0$ depending on $\data$, $\sig$ and $\sig_*$, and applying Lemma \ref{oms}, we obtain $\dashint_B \zeta_\Phi^{\frac{(1+\sig)(1+\sig_*)}{\sig_*-\sig}} \, dx \leq c$, which concludes the proof.
\end{proof}
\begin{rmk}
    Note that the roles of $\om$ and $\ov{\om}$ can be interchanged, and the analogous estimate \eqref{MovMest} in Lemma \ref{MovM} remains valid.
    In particular, if $\Psi(x,\ovom f) \in L^{1+\sig_*}(B)$ and $|\log \om|_{\BMO}$ is sufficiently small, then it follows that $\Psi(x,\om f) \in L^{1+\sig}(B)$.
    Moreover, Lemma \ref{MovM} continues to hold in the special case $\Psi(x,z)=G(|z|)$.
\end{rmk}

Using the previous lemma, we prove a weighted Sobolev–Poincar\'e inequality adapted to the double-phase structure.
\begin{lem} \label{Soblem}
    Let $B=B_r \subset \Omega$ with $r \leq 1$ and $G,H \in \N$ satisfy \eqref{kGHcond}. Then there exists $\delta>0$, depending on $\data$, such that if $|\log \om|_{\BMO} \leq \delta $, the following holds.
    \begin{enumerate}
    \item There exists $\theta(n,\Lambda) \in (0,1)$ such that for any $v \in W^{1,\Psi}_{\om}(2B)$, we have
    \begin{align*}
        \dashint_B \Psi \left( x,\frac{|v-(v)_B|}{r} \om\right) \, dx \leq c \left[ 1+\left( \int_B G(|Dv| \om) \, dx \right)^{\alpha/n}\right] \left[ \dashint_B \Psi(x,| Dv| \om)^\theta \, dx \right]^{1/\theta},
    \end{align*}
    for some $c=c(n, \Lambda, \kappa, i(G),s(G),i(H),s(H), \alpha,
    \norm{a}_{C^{0,\alpha}})>0$.
    \item If $v=0$ on $\partial B \cap B_\rho (y) \neq \emptyset$, where $B_\rho(y)$ satisfies $y \in B$ and $|B_\rho(y) \setminus B| \geq \nu |B_\rho|$ for some $\nu>0$, then there exists $\theta(n,\Lambda) \in (0,1)$ which satisfies
    \begin{align*}
        \dashint_{B_\rho (y) \cap B} \Psi\left(x, \frac{ |v|}{\rho} \om \right) \, dx \leq c \left[ 1+\left( \int_B G(|Dv| \om) \, dx \right)^{\alpha/n}\right] \left[ \dashint_{B_\rho (y) \cap B} \Psi(x,| Dv| \om)^\theta \, dx \right]^{1/\theta},
    \end{align*}
    for some $c=c( n, \Lambda, \kappa, i(G),s(G),i(H),s(H), \alpha,
    \norm{a}_{C^{0,\alpha}}, \nu)>0$.
    \end{enumerate}
\end{lem}
\begin{proof}
We provide the proof for (1); the boundary case (2) follows by a similar argument.
First, we claim that there exist $\theta_0 \in (0,1)$ and $c>1$ depending on $\data$, such that
\begin{align} \label{Sobclaim}
    \dashint_{B_r} G^{1+\frac{\alpha}{n}} \left( \frac{|v-(v)_{B_r}|}{r}\om\right) \, dx \leq c \left( \dashint_{B_r} G^{\theta_0}(|Dv|\om) \, dx\right)^{\frac{1}{\theta_0}(1+\frac{\alpha}{n})}.
\end{align}
This claim was established in the proof of \cite[Theorem 4.2]{Baasandorj20} when $\om=1$.
Hence, it also holds for any constant weight.
Using \cite[Lemma 1.2.2]{Kokilashvili91}, there exists $\theta_1 \in (0,1)$ close to 1 such that $G^{\theta_1} \in \N$.
Applying Lemma \ref{MovM} with sufficiently small $\delta>0$, we find that
\begin{align*}
    \dashint_{B_r} G^{1+\frac{\alpha}{n}} \left( \frac{|v-(v)_{B_r}|}{r}\om\right) \, dx \leq c  \left(\dashint_{B_r} G^{\left(1+\frac{\alpha}{n}\right)\theta_1} \left( \frac{|v-(v)_{B_r}|}{r}\ov{\om}\right)  \, dx\right)^{\frac{1}{\theta_1}}.
\end{align*}
Applying the claim \eqref{Sobclaim} to the constant weight $\ov{\om}$ and the Young function $G^{\theta_1} \in \N$, there exists $\theta_2 \in (0,1)$ depending on $\data$ and $\theta_1$ such that
\begin{align*}
    \left(\dashint_{B_r} G^{\left(1+\frac{\alpha}{n}\right)\theta_1} \left( \frac{|v-(v)_{B_r}|}{r}\ov{\om}\right) \right)^{\frac{1}{\theta_1}} \, dx \leq c \left( \dashint_{B_r} G^{\theta_1\theta_2}(|Dv|\ov{\om})  \, dx \right)^{\frac{1}{\theta_1\theta_2}(1+\frac{\alpha}{n})}.
\end{align*}
Applying Lemma \ref{MovM} once more with small enough $\delta>0$, we get
\begin{align*}
     \left( \dashint_{B_r} G^{\theta_1\theta_2}(|Dv|\ov{\om})\right)^{\frac{1}{\theta_1\theta_2}(1+\frac{\alpha}{n})} \, dx \leq c \left( \dashint_{B_r} G^{\theta_1^2\theta_2}(|Dv|\om) \, dx \right)^{\frac{1}{\theta_1^2\theta_2}(1+\frac{\alpha}{n})}.
\end{align*}
Therefore, the claim \eqref{Sobclaim} holds with $\theta_0=\theta_1^2\theta_2 \in (0,1)$.

We complete the proof by considering two cases depending on the behavior of $a(x)$.
Suppose that 
\begin{align} \label{sobcase}
    \sup_{B_r} a(x) \leq 4[a]_{\alpha} r^\alpha.
\end{align}
By applying the claim \eqref{Sobclaim} and H\"older inequality, we obtain
\begin{align*}
    \dashint_B \Psi \left( x,\frac{|v-(v)_B|}{r} \om\right) \, dx &\leq c  \dashint_B G \left( \frac{|v-(v)_B|}{r} \om\right) \, dx + cr^\alpha \dashint_B G^{1+\frac{\alpha}{n}} \left( \frac{|v-(v)_B|}{r} \om\right)  \, dx\\
    & \leq c  \left[ 1+ r^\alpha \left( \dashint_B G^{\theta_0} \left( |Dv| \om\right)  \, dx \right)^{\frac{\alpha}{\theta_0n}} \right] \left( \dashint_B G^{\theta_0} \left( |Dv| \om\right) \, dx\right)^{\frac{1}{\theta_0}} \\
    &\leq c  \left[ 1+ \left( \int_B G \left( |Dv| \om\right)  \, dx \right)^{\frac{\alpha}{n}} \right] \left( \dashint_B \Psi^{\theta_0} \left( x,|Dv| \om\right) \, dx\right)^{\frac{1}{\theta_0}}.
\end{align*}
Conversely, suppose that \eqref{sobcase} fails.
Then there exists $y \in B_r$ such that
\begin{align*}
    a(y) > 4[a]_{\alpha} r^\alpha.
\end{align*}
Then we have $\frac{a(y)}{2}\leq a(x) \leq 2a(y)$, which implies $\frac{\Psi(y,z)}{2}\leq \Psi(x,z) \leq 2\Psi(y,z)$.
Therefore, using Lemma \ref{Orl-Sob} with $\Phi(t)= \Psi(y,t)$, there exists $\ov\theta_0 \in (0,1)$ depending on $\data$ such that
\begin{align*}
    \dashint_B \Psi \left( x,\frac{|v-(v)_B|}{r} \om\right) \, dx &\leq c \dashint_B \Psi \left( y,\frac{|v-(v)_B|}{r} \om   \right)  \, dx \\
    &\leq c\left(\dashint_B \Psi(y,| Dv| \om)^{\ov\theta_0} \, dx \right)^{\frac{1}{\ov\theta_0}}  \leq  c\left(\dashint_B \Psi(x,| Dv| \om)^{\ov\theta_0}  \, dx \right)^{\frac{1}{\ov\theta_0}}.
\end{align*}
Combining both cases and setting $\theta = \max\{\theta_0,\ov\theta_0\} \in (0,1)$, we arrive at the desired result.
\end{proof}

\section{Properties of some reference problems} \label{sec4}

In this section, we present several properties of the reference problems that will be used in our proof.
Let $B \Subset \Omega$ be a ball, $G, H \in \N$ satisfy \eqref{GHcond}, and a weight $\M$ satisfy \eqref{Mcond} with $|\log\M|_{\BMO(B')}\leq\delta_0$.
We set
\begin{align*}
    \Psi_\om(x,z) = G(\om(x)|z|) + a(x)H(\om(x)|z|). 
\end{align*}
Note that $\Psi_\om(x,z) = \Psi(x, \om(x)z)$, and $\Psi_\om(x,z) \approx \Psi(x,\M z)$ by \eqref{Mwcomp}.

We define the Musielak-Orlicz space $L^{\Psi_\om}(\Omega)$ as the set of all measurable functions $f \in L^1(\Omega)$ such that
\begin{align*}
    \int_\Om \Psi_\om(x,|f(x)|) \, dx < \infty. 
\end{align*}
Then by virtue of Lemma \ref{oms}, this space becomes a reflexive Banach space when equipped with the Luxembourg norm,
\begin{align*}
    \norm{f}_{L^{\Psi_\om}(\Om)}= \inf_{\lambda>0} \left\{\int_\Omega \Psi_\om\left(x,\frac{|f(x)|}{\lambda}\right) \, dx \leq 1\right\},
\end{align*}
provided that $|\log \om|_{BMO(\Omega)} \leq \delta$ for a sufficiently small $\delta>0$ depending on $\data$.
We also define weighted Orlicz-Sobolev space $W^{1,\Psi_\om}(\Omega)$ as the set of all functions $f \in W^{1,1}(\Omega)$ such that $f, |Df| \in L^{\Phi_\om}(\Omega)$,  with the norm $\norm{f}_{W^{1,\Psi_\om}(\Omega)} = \norm{f}_{L^{\Psi_\om}(\Omega)} + \norm{Df}_{L^{\Psi_\om}(\Omega)}$.
Then $W^{1,\Psi_\om}(\Omega)$ can be approximated by smooth functions by Lemma \ref{Lav}.
Therefore, it is possible to define $W^{1,\Psi_\om}_0(\Om)$ as the closure of $C^\infty_0(\Om)$ in $W^{1,\Psi_\om}(\Om)$.
For further details on Musielak-Orlicz spaces and their corresponding Sobolev spaces, we refer the reader to \cite{Harjulehto19,Musielak83,Harjulehto16,Benkirane14}.

\subsection{Reference problem 1.}
We consider the Dirichlet problem:
\begin{align} \label{ref1}
    \begin{cases}
        \dv\left(\M A(x,\M Dh)\right) =0 \quad \text{ in } B,\\
        h=h_0\quad \text{ on } \partial B,
    \end{cases}
\end{align}
where $h_0 \in W^{1,1}(B)$ is such that $\Psi(x,\M Dh_0) \in L^1(B)$.
Using the monotonicity method in Musielak--Orlicz space, one can establish the existence of a unique solution $h$ such that $h-h_0 \in W^{1,\Psi_\om}_{0}(B)$.
\begin{lem} \label{ref1e}
    Let $h$ be the solution of \eqref{ref1}. Then the following properties hold:
    \begin{enumerate}
        \item There exists a constant $c>0$ depending on $\data$, such that
        \begin{align*}
        \int_B \Psi(x, \M Dh) \, dx \leq c\int_B \Psi(x,\M Dh_0) \, dx. 
        \end{align*}
        \item There exist  $\sigma_1>0$ and $c>0$ depending on $\data$ and $\norm{H(x,\M Dh)}_{L^1(B)}$ such that $H(x,\M Dh) \in L^{1+\sigma_1}_{\loc}(B)$ and
        \begin{align*}
            \dashint_{\frac{1}{2}B} \Psi(x, \M Dh)^{1+\sigma_1} \, dx \leq c\left(\dashint_B \Psi(x,\M Dh) \, dx\right)^{1+\sigma_1}.
        \end{align*}
        \item If $H(x,\M Dh_0) \in L^{1+\sigma_*}(B)$ for some $\sigma_*>0$, then there exist $\sigma \in (0,\sigma_*)$ and $c>0$ depending on $\data$, $\sigma_*$, and $\norm{H(x,\M Dh)}_{L^1(B)}$ such that $H(x,\M Dh) \in L^{1+\sigma}(B)$ and
        \begin{align*}
            \dashint_{B} \Psi(x, \M Dh)^{1+\sigma} \, dx \leq c\dashint_B \Psi(x,\M Dh_0)^{1+\sigma} \, dx.
        \end{align*}
    \end{enumerate}
\end{lem}
\begin{proof}
(1) We take $h-h_0 \in W^{1,\Psi_\om}_{0}(B)$ as a test function in \eqref{ref1}.
Then 
\begin{align*}
    \int_B \inner{A(x,\M Dh)}{\M Dh} \, dx = \int_B \inner{A(x,\M Dh)}{\M Dh_0} \, dx.
\end{align*}
Applying \eqref{AGHcond} and Young's inequality \eqref{young}, we obtain
\begin{align*}
    \int_B \Psi(x, \M Dh) \, dx&\leq c \int_B \left( G'(\M Dh) + a(x) H'(\M Dh)\right)|\M Dh_0| \, dx \\
    &\leq c\eps\int_B \Psi(x, \M Dh) \, dx + c(\eps)\int_B \Psi(x,\M Dh_0) \, dx,
\end{align*}
which implies (1) by choosing small $\eps>0$.

(2) Fix $y\in B$ such that $B_{2r}(y) \subset B$.
Let $\eta \in C^\infty_0(B_{2r}(y))$ be a cutoff function satisfying $\1_{B_r(y)} \leq \eta \leq \1_{B_{2r}(y)}$ and $|D\eta| \leq 2/r$.
Taking $\eta^{s}(h-(h)_{B_{2r}(y)})$ with $s=\max\{s(G),s(H)\}+1$ as a test function in \eqref{ref1}, and applying Young's inequality \eqref{young}, we have
\begin{align*}
    \int_{B_{2r}(y)} \eta^s \Psi(x,\M Dh) \, dx &\leq c \int_{B_{2r}(y)} \eta^{s-1}  (G'(\M Dh) + a(x) H'(\M Dh)) \frac{|h-(h)_{B_{2r}(y)}|}{r}\om \, dx \\
    &\leq c \int_{B_{2r}(y)} \eta^{s-1} \left( (\eps\eta) G(\M Dh)  + \frac{1}{(\eps \eta)^{s(G)}}G\left( \frac{|h-(h)_{B_{2r}(y)}|}{r}\om\right)\right)\,dx\\
    &+ c \int_{B_{2r}(y)} a(x)\eta^{s-1} \left( (\eps\eta) H(\M Dh)  + \frac{1}{(\eps \eta)^{s(H)}}H\left( \frac{|h-(h)_{B_{2r}(y)}|}{r}\om\right)\right) \,dx \\
    &\leq c\eps \int_{B_{2r}(y)} \eta^{s} \Psi(x,\M Dh)\,dx +c(\eps)\int_{B_{2r}(y)} \Psi\left(x,\frac{|h-(h)_{B_{2r}(y)}|}{r}\om \right) \, dx.
\end{align*}
Choosing small $\epsilon>0$, we obtain
\begin{align*}
     \int_{B_{2r}(y)} \eta^s \Psi(x,\M Dh) \, dx\leq c\int_{B_{2r}(y)} \Psi\left(x,\frac{|h-(h)_{B_{2r}(y)}|}{r}\om \right) \, dx.
\end{align*}
Applying the weighted Sobolev-Poincar\'e inequality (Lemma \ref{Soblem}), we find
\begin{align*}
     \dashint_{B_{2r}(y)} \Psi\left(x,\frac{|h-(h)_{B_{2r}(y)}|}{r}\om \right) \, dx \leq c \left(\dashint_{B_{2r}(y)} \Psi(x,\M Dh)^\theta \, dx \right) ^{1/\theta}
\end{align*}
for some $\theta \in (0,1)$.
Combining these estimates yields the reverse H\"older inequality
\begin{align*}
     \dashint_{B_{r}(y)} \Psi(x,\M Dh) \, dx \leq c \left(\dashint_{B_{2r}(y)} \Psi(x,\M Dh)^\theta \, dx \right) ^{1/\theta}.
\end{align*}
Finally, the desired higher integrability result (2) follows from a direct application of Gehring's Lemma.

(3) Fix $y \in B$ such that the ball $B_{2r}(y)$ satisfies $|B_{2r}(y) \setminus B| > |B_{2r}(y)| /10$.
Let $\eta \in C^\infty_0(B_{2r}(y))$ be a cutoff function such that $\1_{B_r(y)} \leq \eta \leq \1_{B_{2r}(y)}$ and $|D\eta| \leq 2/r$.
We take $\eta^s(h-h_0)$ with $s=\max\{s(G),s(H)\}+1$ as a test function in \eqref{ref1}.
This yields
\begin{align*}
    \int_{B \cap B_{2r}(y)} \eta^s \Psi(x, \M Dh) \, dx &\leq c \int_{B \cap B_{2r}(y)} \eta^{s-1}  (G'(\M Dh) + a(x) H'(\M Dh)) \frac{|h-h_0|}{r}\om \, dx \\
    &+ c \int_{B \cap B_{2r}(y)} \eta^{s} \Psi(x, \M Dh_0) \, dx .
\end{align*}
Following the same argument used in the proof of (2), we obtain
\begin{align*}
    \int_{B \cap B_{2r}(y)} \eta^s \Psi(x, \M Dh) \, dx &\leq c \int_{B \cap B_{2r}(y)}  \Psi\left(x,\frac{|h-h_0|}{r}\om \right) \, dx + c \int_{B \cap B_{2r}(y)} \eta^{s} \Psi(x, \M Dh_0) \, dx.
\end{align*}
By applying the weighted Sobolev-Poincar\'e inequality (Lemma \ref{Soblem}) and recalling the assertion (1), 
we have
\begin{align*}
   \dashint_{B \cap B_{2r}(y)}  \Psi\left(x,\frac{|h-h_0|}{r}\om \right) \, dx  &\leq  c \left(\dashint_{B \cap B_{2r}(y)} \Psi(x, \M(Dh- Dh_0))^\theta  \, dx\right)^{1/\theta} \\
   &\leq c \left(\dashint_{B \cap B_{2r}(y)} \Psi(x, \M Dh)^\theta  \, dx\right)^{1/\theta} + c\dashint_{B \cap B_{2r}(y)} \Psi(x, \M Dh_0)  \, dx.
\end{align*}
Therefore, we arrive at the reverse H\"older inequality
\begin{align*}
    \dashint_{B_{r}(y)} \Psi(x, \M Dh)\1_B  \, dx \leq c \left(\dashint_{B_{2r}(y)} (\Psi(x, \M Dh) \1_B)^\theta \, dx\right)^{1/\theta} + c\dashint_{B_{2r}(y)} \Psi(x, \M Dh_0) \1_B \, dx.
\end{align*}
The desired higher integrability result (3) follows from Gehring's Lemma.
\end{proof}
We use the following notation:
\begin{align*}
    \overline{\M} = \Log{\M}_B, \quad \text{ and } \quad \overline{\om} = \Log{\om}_B.
\end{align*}

\subsection{Reference problem 2.}
We consider the Dirichlet problem with the averaged matrix weight
\begin{align} \label{ref2}
    \begin{cases}
        \dv\left(\ov{\M} A(x,\ov{\M} Dk)\right) =0 \quad \text{ in } B,\\
        k=k_0\quad \text{ on } \partial B,
    \end{cases}
\end{align}
where $k_0 \in W^{1,1}(B)$ satisfies $\Psi(x,\ovM Dk_0) \in L^{1+\sig_0}(B)$ for some $\sig_0>0$, and
\begin{align*}
    \norm{\Psi(x,\ovM Dk_0)}_{L^1(B)} \leq c_0,
\end{align*}
for some $c_0>0$.
Then the following regularity properties can be established using arguments similar to those in Lemma \ref{ref1e}.
\begin{lem} \label{ref2e}
    Let $k$ be the solution to \eqref{ref2}.
    Then the following properties hold.
    \begin{enumerate}
        \item There exists a constant $c>0$ depending on $\data$ such that
        \begin{align*}
        \int_B \Psi(x, \ovM Dk) \, dx \leq c\int_B \Psi(x,\ovM Dk_0) \, dx. 
        \end{align*}
        \item There exist $\sigma \in (0,\sigma_0)$ and $c>0$ depending only on $\data$, $\sigma_0$, and $\norm{H(x,\ovM Dk)}_{L^1(B)}$ such that $H(x,\ovM Dk) \in L^{1+\sigma}(B)$ and
        \begin{align*}
            \dashint_{\frac{1}{2}B} \Psi(x, \ovM Dk)^{1+\sigma} \, dx &\leq c\left(\dashint_B \Psi(x,\ovM Dk) \, dx\right)^{1+\sigma},\\
            \dashint_{B} \Psi(x, \ovM Dk)^{1+\sigma} \, dx &\leq c\dashint_B \Psi(x,\ovM Dk_0)^{1+\sigma} \, dx.
        \end{align*}
    \end{enumerate}
\end{lem}
We now introduce the normalized solution $\til{k}$ defined by
\begin{align*}
    \til{k}(x) = |\ovM|k(x).
\end{align*}
Then $\til{k}$ is the solution of the following Dirichlet problem
\begin{align} \label{ref22}
    \begin{cases}
        \dv \tilde{A}(x, D\til{k}) =0 \quad \text{ in } B,\\
        k= \til{k}_0 :=|\ovM|k_0\quad \text{ on } \partial B,
    \end{cases}
\end{align}
where $\tilde{A}(x,z) = \frac{\ovM}{|\ovM|} A\left(x,\frac{\ovM}{|\ovM|}z\right)$.
Observe that due to \eqref{Mlog}, $\tilde{A}$ satisfies the structural assumptions in \eqref{AGHcond}.
Furthermore, since $|D\til{k}_0| \approx |\ovM Dk_0|$, we have $\Psi(x, D\til{k}_0) \in L^{1+\sig_0}(B)$ and 
\begin{align*}
    \norm{\Psi(x,D\til{k}_0)}_{L^1(B)} \leq c_0.
\end{align*}
Equation \eqref{ref22} belongs to the class of standard Orlicz double-phase problems. Consequently, we can directly apply the reverse H\"older-type inequality established in \cite{Baasandorj20}.
\begin{lem}[{\cite[Theorem 7.1]{Baasandorj20}}] \label{aGrev}
    Let $B=B_{2r}$ and $\til{k}$ be the solution to \eqref{ref22}. Assume that 
    \begin{align*}
        \sup_{B_{r}}a(x) \leq K[a]_\alpha r^\alpha
    \end{align*}
    for some $K \geq 1$.
    Then there exists $c\geq1$ depending on $\data$, $K$ and $ \norm{\Psi(x,D\til{k})}_{L^1(B)}$, such that 
    \begin{align*}
        \left(\dashint_{B_r} G(|D\til{k}|)^{1+\frac{\alpha}{n}} \, dx \right)^{\frac{n}{n+\alpha}} \leq c \dashint_{B_{2r}} \Psi(x,D\til{k}) \, dx.
    \end{align*}
\end{lem}
\subsection{Reference problem 3.}
We consider the following Dirichlet problem with frozen coefficients and an averaged matrix weight
\begin{align} \label{ref2}
    \begin{cases}
        \dv\left(\ov{\M} A(x_0,\ov{\M} Dv)\right) =0 \quad \text{ in } B,\\
        v=v_0\quad \text{ on } \partial B,
    \end{cases}
\end{align}
where $x_0 \in B$ is a fixed point.
Setting $a_0 = a(x_0)$, the following energy estimate holds:
\begin{align} \label{ref3e} 
    \dashint_B G(|\ovM Dv|) + a_0 H(|\ovM Dv|) \, dx \leq c\dashint_B G(|\ovM Dv_0|) + a_0 H(|\ovM Dv_0|) \, dx.
\end{align}
As before, we introduce the normalized solution
\begin{align*}
    \til{v}(x) = |\ovM|v(x).
\end{align*}
which satisfies:
\begin{align} \label{ref32}
    \begin{cases}
        \dv \tilde{A}(x_0, D\til{v}) =0 \quad \text{ in } B,\\
        v= \til{v}_0 :=|\ovM|v_0\quad \text{ on } \partial B,
    \end{cases}
\end{align}
where $\tilde{A}(x_0,z) = \frac{\ovM}{|\ovM|} A\left(x_0,\frac{\ovM}{|\ovM|}z\right)$.
Since \eqref{ref32} is an autonomous elliptic equation with Orlicz growth, we can apply the Lipschitz regularity estimates from \cite[Theorem 1.2]{Lieberman91} to obtain the following lemma.
\begin{lem}[\cite{Lieberman91}] \label{vLip}
    Let $\til{v}$ be the solution to \eqref{ref32}, Then there exists $c\geq 1$ depending on $\data$ such that
    \begin{align*}
        \sup_{\frac{1}{2}B} \{ G(D\til{v})+a_0H(D\til{v})\} \leq c \dashint_B G(D\til{v})+a_0H(D\til{v})\, dx.
    \end{align*}
\end{lem}

\section{Proof of Main result} 
\label{sec5}
We are now ready to give the proof of the main theorem, following the approach in \cite{Mingione07}.
\begin{proof}[Proof of Theorem \ref{Main}]
\textit{Step 1 : Exit time argument.}
Let $B_R \Subset \Om$ with $R \leq r_0$ where $r_0$ will be determined later.
We select $r_1,r_2$ such that $R/2 \leq r_1<r_2 \leq R$.
For $R/2\leq s \leq R$ and $\lambda>0$, we define the upper level set by
\begin{align*}
    E(s,\lambda) = \{x \in B_s : \Psi(x,\M(x)Du(x)) > \lambda\}.
\end{align*}
Then for almost every $x_0 \in E(r_1,\lambda)$, we have
\begin{align*}
    \lim_{r\rightarrow 0} \dashint_{B_r(x_0)} \Psi(x,\M Du) + \frac{1}{\delta} \Psi(x, \M F)\, dx > \lambda.
\end{align*}
Moreover, for almost every $x_0 \in E(r_1,\lambda)$ and any radius $\rho \in [\frac{r_2-r_1}{40},r_2-r_1]$, the following upper bound holds.
\begin{align*}
    \dashint_{B_\rho(x_0)} \Psi(x, \M Du) + \frac{1}{\delta} \Psi(x,\M F) \, dx \leq \left( \frac{40R}{r_2-r_1}\right)^n \dashint_{B_R} \Psi(x, \M Du) + \frac{1}{\delta} \Psi(x,\M F) \, dx =: \lambda_0.
\end{align*}
Thus, by choosing $\lambda> \lambda_0$, for almost every $x_0 \in E(r_1,\lambda)$, there exists a radius $r_{x_0} \in (0,\frac{r_2-r_1}{40})$ such that
\begin{align} \label{rx0}
\begin{cases}
    \dashint_{B_{r_{x_0}}(x_0)} \Psi(x, \M Du) + \frac{1}{\delta} \Psi(x,\M F) \, dx =\lambda, \\
    \dashint_{B_{r}(x_0)} \Psi(x, \M Du) + \frac{1}{\delta} \Psi(x,\M F) \, dx < \lambda \text{ for any } r \in (r_{x_0},r_2-r_1].
\end{cases}
\end{align}
Applying Vitali's covering lemma, we have countable family of pairwise disjoint balls $B_{r_{x_i}}(x_i)$ satisfying \eqref{rx0} and 
\begin{align*}
    E(r_1,\lambda) \subset \bigcup_{i \in \mathbb{N}} B_{5r_{x_i}}(x_i) \subset B_{r_2},
\end{align*}
except some negligible set.
For simplicity, we denote $r_i = r_{x_i}$ and $aB_i = B_{ar_{x_i}}(x_i)$.
Since $40r_i \leq r_2-r_1 \leq R$, we obtain
\begin{align} \label{lamcov}
\begin{cases}
    \dashint_{B_i} \Psi(x, \M Du) + \frac{1}{\delta} \Psi(x,\M F) \, dx =\lambda, \\
    \dashint_{40B_i} \Psi(x, \M Du) + \frac{1}{\delta} \Psi(x,\M F) \, dx < \lambda.
\end{cases}
\end{align}

\textit{Step 2 : First comparison estimate.}
For each ball $40B_i$, we consider the weak solution $h_i$ to the following Dirichlet problem
\begin{align} \label{hipde}
    \begin{cases}
    \dv\left(\M A(x,\M Dh_i)\right) =0 \quad \text{ in } 40B_i,\\
    h_i=u\quad \text{ on } \partial 40B_i.
    \end{cases}
\end{align}
By the energy estimate in Lemma \ref{ref1e}, we have
\begin{align} \label{hi1}
    \dashint_{40B_i} \Psi(x, \M Dh_i) \,dx \leq c \dashint_{40B_i} \Psi(x, \M Du) \,dx.
\end{align}
We set $\ovM = \Log{\M}_{20B_i} $.
Then, applying Lemma \ref{MovM}, the higher integrability from Lemma \ref{ref1e}, and \eqref{hi1}, we obtain for some $\sig_1>0$:
\begin{align} \label{hi2}
    \dashint_{20B_i} \Psi(x, \ovM Dh_i)^{1+\frac{\sig_1}{2}} \, dx &\leq c \left(\dashint_{20B_i} \Psi(x, \M Dh_i)^{1+\sig_1} \, dx \right)^{(1+\frac{\sig_1}{2})\frac{1}{1+\sig_1}} \nonumber\\
    &\leq c \left(\dashint_{40B_i} \Psi(x, \M Dh_i)\, dx \right)^{1+\frac{\sig_1}{2}} \nonumber \\
    &\leq c \left(\dashint_{40B_i} \Psi(x, \M Du)\, dx \right)^{1+\frac{\sig_1}{2}},
\end{align}
provided that $\delta>0$ is sufficiently small.
We now establish a comparison estimate between $u$ and $h$.
Using $\varphi = u-h_i \in W^{1,\Psi_\om}_0(40B_i)$ as a test function to \eqref{mainpde} and \eqref{hipde}, we arrive at the following identity:
\begin{align*}
    \dashint_{40B_i}  \inner{A(x, \M Du)- A(x, \M Dh_i)}{\M Du- \M Dh_i}\, dx = \dashint_{40B_i} \inner{B(x,\M F)}{\M Du- \M Dh_i} \, dx.
\end{align*}
Then using \eqref{VGH}, \eqref{young} and \eqref{hi1}, we find that for any $\epsilon \in (0,1)$,
\begin{align} \label{ref1V}
    \dashint_{40B_i}&  |V_G(\M Du) - V_G(\M Dh_i)|^2 + a(x)|V_H(\M Du) - V_H(\M Dh_i)|^2 \, dx \nonumber\\
    &\leq c \dashint_{40B_i}(G'(|\M F|) + a(x) H'(|\M F|))(|\M Du|+ |\M Dh_i|) \, dx \nonumber \\
    &\leq \eps \dashint_{40B_i}\Psi(x, \M Dh_i) + \Psi(x, \M Du) \, dx + c_\eps \dashint_{40B_i} \Psi(x, \M F) \, dx \nonumber \\
    &\leq c\eps \dashint_{40B_i} \Psi(x, \M Du) \, dx + c_\eps \dashint_{40B_i} \Psi(x, \M F) \, dx.
\end{align}

\textit{Step 3 : Second comparison estimate.} 
We now consider the following Dirichlet problem
\begin{align} \label{kipde}
    \begin{cases}
    \dv\left(\ovM A(x,\ovM Dk_i)\right) =0 \quad \text{ in } 20B_i,\\
    k_i=h_i\quad \text{ on } \partial 20B_i,
    \end{cases}
\end{align}
where $\ovM =\Log{\M}_{20B_i}$.
Then by Lemma \ref{ref2e} and \eqref{hi2}, we get
\begin{align} \label{ki1}
    \dashint_{20B_i} \Psi(x, \ovM Dk_i) \,dx &\leq c \dashint_{20B_i} \Psi(x, \ovM Dh_i) \,dx \leq c \dashint_{40B_i} \Psi(x, \M Du) \,dx, \\
    \dashint_{20B_i} \Psi(x, \ovM Dk_i)^{1+\sig_2} \,dx &\leq c \dashint_{20B_i} \Psi(x, \ovM Dh_i)^{1+\sig_2} \,dx \leq c \left(\dashint_{40B_i} \Psi(x, \M Du) \,dx\right)^{1+\sig_2}. \label{ki2}
\end{align}
for some $\sig_2 \in (0,\frac{\sig_1}{2})$.
Moreover, by choosing small enough $\delta>0$, Lemma \ref{MovM} ensures that $\Psi(x, \M Dk_i) \in L^1(20B_i)$.
This justifies the use of $\varphi = h_i-k_i \in W^{1,\Psi_{\om}}_0 \cap W^{1,\Psi_{\ovom}}_0(20B_i)$ as a test function to \eqref{hipde} and \eqref{kipde}, which leads to the following identity
\begin{align*}
    &\dashint_{20B_i}  \inner{A(x, \M Dh_i)- A(x, \ovM Dk_i)}{\M Dh_i- \ovM Dk_i}\, dx \\
    = &\dashint_{20B_i}\inner{A(x, \M Dh_i)}{(\M - \ovM) Dk_i}\, dx - \dashint_{20B_i}\inner{A(x, \ovM Dk_i)}{(\M - \ovM) Dh_i}\, dx.
\end{align*}
Using \eqref{VGH}, \eqref{young} and \eqref{ki1}, we obtain for any $\epsilon \in (0,1)$
\begin{align*}
    \dashint_{20B_i}&  |V_G(\M Dh_i) - V_G(\ovM Dk_i)|^2 + a(x)|V_H(\M Dh_i) - V_H(\ovM Dk_i)|^2 \, dx \\
    &\leq c \dashint_{20B_i}(G'(|\M Dh_i|) + a(x) H'(|\M Dh_i|))|(\M - \ovM)Dk_i| \, dx  \\
    &\quad +c \dashint_{20B_i}(G'(|\ovM Dk_i|) + a(x) H'(|\ovM Dk_i|))|(\M - \ovM)Dh_i| \, dx  \\
    &\leq \eps \dashint_{20B_i}\Psi(x, \M Dh_i) \, dx + c_\eps \dashint_{20B_i} G\left(\frac{|\M-\ovM|}{|\ovM|} |\ovM Dk_i|\right) +a(x)H\left(\frac{|\M-\ovM|}{|\ovM|} |\ovM Dk_i|\right) \, dx  \\
    &\quad+ \eps \dashint_{20B_i}\Psi(x, \ovM Dk_i) \, dx + c_\eps \dashint_{20B_i} G\left(\frac{|\M-\ovM|}{|\ovM|} |\ovM Dh_i|\right) +a(x)H\left(\frac{|\M-\ovM|}{|\ovM|} |\ovM Dh_i|\right) \, dx \\
    &\leq c\eps \dashint_{40B_i}\Psi(x, \M Du) \, dx + c_\eps \dashint_{20B_i} \zeta(x)\Psi(x,\ovM Dk_i) +\zeta(x)\Psi(x,\ovM Dh_i)\, dx,
\end{align*}
where
\begin{align*}
    \zeta(x) =  \left(\frac{|\M-\ov{\M}|}{|\ovM|} \right)^{1+i(G)} +  \left(\frac{|\M-\ov{\M}|}{|\ovM|} \right)^{1+s(G)} +\left(\frac{|\M-\ov{\M}|}{|\ovM|} \right)^{1+i(H)} + \left(\frac{|\M-\ov{\M}|}{|\ovM|} \right)^{1+s(H)}.
\end{align*}
By choosing $\delta>0$ sufficiently small  and applying Lemma \ref{M-ovM} and \eqref{ki2}, we obtain
\begin{align*}
    \dashint_{20B_i} \zeta(x)\Psi(x,\ovM Dk_i)\, dx &\leq \left(\dashint_{20B_i} \zeta(x)^{\frac{1+\sig_2}{\sig_2}} \, dx\right)^{\frac{\sig_2}{1+\sig_2}}\left(\dashint_{20B_i} \Psi(x,\ovM Dk_i)^{1+\sig_2}\, dx\right)^{\frac{1}{1+\sig_2}} \\
    &\leq c\delta^{1+i_1}\dashint_{40B_i}\Psi(x, \M Du) \, dx,
\end{align*}
where $i_1 = \min \{i(G),i(H)\}$. 
Consequently, we have for any $\epsilon \in (0,1)$,
\begin{align} \label{ref2V}
    \dashint_{20B_i}&  |V_G(\M Dh_i) - V_G(\ovM Dk_i)|^2 + a(x)|V_H(\M Dh_i) - V_H(\ovM Dk_i)|^2 \, dx \nonumber \\
    &\leq (c_1\eps + c_\eps \delta^{1+i_1})\dashint_{40B_i}\Psi(x, \M Du) \, dx.
\end{align}

\textit{Step 4 : Third comparison estimate.} 
We choose a point $x_i \in \ov{10B_i}$ such that $a(x)$ attains its supremum
\begin{align*}
    a_i = a(x_i) = \sup_{x \in 10B_i} a(x).
\end{align*}
We then consider the following Dirichlet problem
\begin{align*}
    \begin{cases}
    \dv\left(\ovM A(x_i,\ovM Dv_i)\right) =0 \quad \text{ in } 10B_i,\\
    v_i=k_i\quad \text{ on } \partial 10B_i.
    \end{cases}
\end{align*}
Then the normalized solutions $\til{k}_i(x) = |\ovM|k_i(x)$ and $\til{v}_i(x) = |\ovM|v_i(x)$ satisfy the following equations
\begin{align} \label{tilkivipde}
    \begin{cases}
    \dv\til{A}(x, D\til{k}_i) =0 \quad \text{ in } 20B_i,\\
    \til{k}_i=|\ovM|h_i\quad \text{ on } \partial 20B_i,
    \end{cases} \quad 
    \begin{cases}
    \dv\til{A}(x_i, D\til{v}_i) =0 \quad \text{ in } 10B_i,\\
    \til{v}_i=\til{k}_i\quad \text{ on } \partial 10B_i,
    \end{cases}
\end{align}
where $\tilde{A}(x,z) = \frac{\ovM}{|\ovM|} A\left(x,\frac{\ovM}{|\ovM|}z\right)$.
Then \eqref{ki1} and \eqref{ki2} implies
\begin{align} \label{tilk}
    \dashint_{20B_i} \Psi(x,  D\til{k}_i) \,dx \leq c \dashint_{40B_i} \Psi(x, \M Du) \,dx, \quad \dashint_{20B_i} \Psi(x,  D\til{k}_i)^{1+\sig_2} \,dx \leq c \left(\dashint_{40B_i} \Psi(x, \M Du) \,dx\right)^{1+\sig_2},
\end{align}
since $|D\til{k_i}| \approx |\ovM Dk_i|$. Moreover,  by \eqref{ref3e} we have
\begin{align} \label{vk}
    \dashint_{10B_i} \Psi(x_i, D\til{v}_i) \,dx &\leq c \dashint_{10B_i} \Psi(x_i, D\til{k}_i) \,dx. 
\end{align}
Applying $\varphi = \til{v}_i-\til{k}_i \in W^{1,\Psi}_0(10B_i)$ as a test function to \eqref{tilkivipde}, we obtain
\begin{align*}
    \dashint_{10B_i} \inner{\til{A}(x_i, D\til{v}_i)-\til{A}(x_i, D\til{k}_i)}{D\til{v}_i-D\til{k}_i} \,dx = \dashint_{10B_i} \inner{\til{A}(x, D\til{k}_i)-\til{A}(x_i, D\til{k}_i)}{D\til{v}_i-D\til{k}_i} \,dx.
\end{align*}
Therefore, we get
\begin{align*}
    \dashint_{10B_i}&  |V_G(\ovM Dk_i) - V_G(\ovM Dv_i)|^2 + a(x)|V_H(\ovM Dk_i) - V_H(\ovM Dv_i)|^2 \, dx \\
    &\leq c\dashint_{10B_i}  |V_G( D\til{k}_i) - V_G(D\til{v}_i)|^2 + a_i|V_H(D\til{k}_i) - V_H(D\til{v}_i)|^2 \, dx \\
    &\leq c(\osc_{10B_i}a)\dashint_{10B_i}  H'(|D\til{k}_i|)|D\til{v}_i-D\til{k}_i| \, dx = I.
\end{align*}
For a constant $K\geq 20$ to be determined later, we consider two alternative cases.
\begin{align} \label{GHphase}
    &\inf_{x \in 10B_i} a(x) > K[a]_\alpha r_i^\alpha \quad ((G,H)\text{-phase}), \\
    &\inf_{x \in 10B_i} a(x) \leq K[a]_\alpha r_i^\alpha \quad (G\text{-phase}). \label{Gphase}
\end{align}
We first consider the case of $(G,H)$-phase  \eqref{GHphase}.
We note that
\begin{gather*}
    \osc_{10B_i}a \leq 20[a]_\alpha r^\alpha_i \leq \frac{20}{K}a(x),\\
    a(x) \leq a_i \leq a(x) + \osc_{10B_i}a \leq 2a(x).
\end{gather*}
Using these bounds and \eqref{tilk}, we find that
\begin{align*}
    I &\leq \frac{c}{K}\dashint_{10B_i}  a(x)H'(|D\til{k}_i|)|D\til{v}_i-D\til{k}_i| \, dx \leq \frac{c}{K}\dashint_{10B_i}  a(x)H(|D\til{v}_i|) + a(x)H(|D\til{k}_i|) \, dx \\
    &\leq \frac{c}{K}\dashint_{10B_i}  \Psi(x_i, D\til{k}_i) \, dx  \leq \frac{c}{K}\dashint_{10B_i}  \Psi(x, \ovM Dk_i) \, dx  \\
    &\leq \frac{c}{K}\dashint_{40B_i}\Psi(x, \M Du) \, dx.
\end{align*}
Therefore, we obtain
\begin{align*}
    \dashint_{10B_i}&  |V_G(\ovM Dk_i) - V_G(\ovM Dv_i)|^2 + a(x)|V_H(\ovM Dk_i) - V_H(\ovM Dv_i)|^2 \, dx \leq \frac{c_2}{K}\dashint_{40B_i}\Psi(x, \M Du) \, dx.
\end{align*}
We next consider the case of $G$-phase \eqref{Gphase}. 
Observe that
\begin{align} \label{GHa}
a_i \leq 20[a]_\alpha r^\alpha_i + \inf_{x\in10B_i}a(x) \leq    2K[a]_\alpha r^\alpha_i.
\end{align}
Applying Lemma \ref{aGrev}, we get
\begin{align} \label{revHol}
    \left(\dashint_{10B_i} G(|D\til{k}_i|)^{1+\frac{\alpha}{n}} \, dx \right)^{\frac{n}{n+\alpha}} \leq c \dashint_{10B_i} \Psi(x,D\til{k}_i) \, dx.
\end{align}
Applying Young's inequality with, we obtain for any $\tau \in (0,1)$
\begin{align*}
    I & \leq c\dashint_{10B_i}  a(x)H'(|D\til{k}_i|)|D\til{k}_i| \, dx +c\dashint_{10B_i}  a(x)H'(|D\til{k}_i|)|D\til{v}_i| \, dx\\
    &\leq c \dashint_{10B_i}  a(x_i)H(|D\til{k}_i|) \, dx +\tau \dashint_{10B_i}  a(x_i)H(|D\til{v}_i|) \, dx +\frac{c}{\tau^{s(H)}} \dashint_{10B_i}  a(x_i)H(|D\til{k}_i|) \, dx \\
    &\leq c\left( 1+\frac{1}{\tau^{s(H)}}\right) \dashint_{10B_i}  a(x_i)H(|D\til{k}_i|) \, dx +\tau \dashint_{10B_i}  a(x_i)H(|D\til{v}_i|) \, dx.
\end{align*}
Using \eqref{kGHcond}, \eqref{GHa}, \eqref{revHol}, H\"older inequality and \eqref{tilk}, we get
\begin{align} \label{aHDk}
    \dashint_{10B_i}  a(x_i)H(|D\til{k}_i|) \, dx &\leq c_K r^\alpha_i \dashint_{10B_i}  G(|D\til{k}_i|) + G(|D\til{k}_i|)^{1+\frac{\alpha}{n}} \, dx \nonumber\\
    &\leq c_K r^\alpha_i \left(\dashint_{10B_i}  \Psi(x,|D\til{k}_i|)  \, dx  +\left(\dashint_{10B_i}  \Psi(x,|D\til{k}_i|)  \, dx\right)^{1+\frac{\alpha}{n}}\right)\nonumber \\
    &= c_K  \left(  r^\alpha_i+\left(\int_{10B_i}  \Psi(x,|D\til{k}_i|)  \, dx\right)^{\frac{\alpha}{n}}\right) \dashint_{10B_i}  \Psi(x,|D\til{k}_i|)  \, dx\nonumber \\
    &\leq c_K  \left(  r^\alpha_i+ r_i^{\frac{n\sig_2}{1+\sig_2}}\left(\int_{10B_i}  \Psi(x,|D\til{k}_i|)^{1+\sig_2}  \, dx\right)^{\frac{\alpha}{n(1+\sig_2)}}\right) \dashint_{10B_i}  \Psi(x,|D\til{k}_i|)  \, dx \nonumber\\
    &\leq c_K  \left(  r^\alpha_i+ r_i^{\frac{n\sig_2}{1+\sig_2}}\left(\int_{10B_i}  \Psi(x,\M Du)  \, dx\right)^{\frac{\alpha}{n}}\right) \dashint_{10B_i}  \Psi(x,|D\til{k}_i|)  \, dx\nonumber\\
    &\leq c_K r_i^{s_1} \dashint_{10B_i}  \Psi(x,|D\til{k}_i|)  \, dx,
\end{align}
where $s_1=\min\{\alpha,\frac{n\sig_2}{1+\sig_2} \}$ and $c_K>0$ depends on $\data$ and $K$, which may vary from line to line.
Furthermore, using \eqref{vk} and \eqref{aHDk}, we obtain
\begin{align} \label{aHDv}
    \dashint_{10B_i}  a(x_i)H(|D\til{v}_i|) \, dx &\leq c\dashint_{10B_i}  G(|D\til{k}_i|) + a(x_i)H(|D\til{k}_i|) \, dx \nonumber \\
    &\leq c_K(1+ r_i^{s_1}) \dashint_{10B_i}  \Psi(x,|D\til{k}_i|)  \, dx
\end{align}
By setting $\tau = r_i^{s_2}$ with $s_2 = \frac{s_1}{2s(H)}$, and combining \eqref{aHDk} and \eqref{aHDv}, we have
\begin{align*}
    \dashint_{10B_i}&  |V_G(\ovM Dk_i) - V_G(\ovM Dv_i)|^2 + a(x)|V_H(\ovM Dk_i) - V_H(\ovM Dv_i)|^2 \, dx \leq I \\
    &\leq c_K \left(\left( 1+\frac{1}{\tau^{s(H)}}\right)r_i^{s_1} +\tau(1+r_1^{s_1}) \right)\dashint_{10B_i}  \Psi(x,|D\til{k}_i|)  \, dx\\
    &\leq c_K r_i^{s_2}\dashint_{40B_i}  \Psi(x,\M Du)  \, dx.
\end{align*}
Therefore, combining the results in both cases, we obtain
\begin{align} \label{ref3V}
    \dashint_{10B_i}&  |V_G(\ovM Dk_i) - V_G(\ovM Dv_i)|^2 + a(x)|V_H(\ovM Dk_i) - V_H(\ovM Dv_i)|^2 \, dx \leq \left(\frac{c_2}{K} + c_K r_i^{s_2}\right)\dashint_{40B_i}\Psi(x, \M Du) \, dx.
\end{align}
Observe that for both cases, it follows that
\begin{align*}
    \dashint_{10B_i}  \Psi(x_i,D\til{v}_i)  \, dx  \leq c\dashint_{10B_i}  \Psi(x_i,D\til{k}_i)  \, dx \leq c\dashint_{40B_i}  \Psi(x_i,\M Du)  \, dx.
\end{align*}
Applying the Lipschitz estimate (Lemma \ref{vLip}), we get
\begin{align} \label{vLipu}
    \sup_{5B_i}  \Psi(x,\ovM Dv_i)\leq c\sup_{5B_i} \Psi(x_i,D\til{v}_i) \leq c\dashint_{10B_i}  \Psi(x_i,D\til{v}_i)  \, dx \leq c\dashint_{40B_i}  \Psi(x_i,\M Du)  \, dx.
\end{align}

\textit{Step 5 : Estimates of level sets.} 
Recalling the exit-time condition \eqref{lamcov} and combining the comparison estimates \eqref{ref1V}, \eqref{ref2V}, and \eqref{ref3V}, we obtain
\begin{align} \label{comp}
    \dashint_{10B_i}&  |V_G(\M Du) - V_G(\ovM Dv_i)|^2 + a(x)|V_H(\M Du) - V_H(\ovM Dv_i)|^2 \, dx \leq S\lambda,
\end{align}
where
\begin{align*}
    S=S(\epsilon,\delta,K,r) = c_1\epsilon + c_\epsilon \delta^{i_1} + \frac{c_2}{K} + c_Kr_0^{s_2}.
\end{align*}
Moreover, \eqref{vLipu} and \eqref{lamcov} imply that
\begin{align} \label{vLiplam}
    \sup_{5B_i}  \Psi(x,\ovM Dv_i)\leq  c_l \lambda.
\end{align}
Using \eqref{comp}, \eqref{vLiplam}, and the fact that
\begin{align*}
    \Psi(x,\M Du) \leq 2( |V_G(\M Du) - V_G(\ovM Dv_i)|^2 + a(x)|V_H(\M Du) - V_H(\ovM Dv_i)|^2) +2\Psi(x,\ovM Dv_i),
\end{align*}
we obtain
\begin{align} \label{PsiBi}
    \int_{5B_i \cap \{H(x,\M Du)>4c_l\lambda \}} \Psi(x,\M Du) \, dx \leq 40^nS\lambda|B_i|.
\end{align}
From \eqref{lamcov}, it follows that
\begin{align*}
    |B_i| = \frac{1}{\lambda} \dashint_{B_i} \Psi(x, \M Du) + \frac{1}{\delta} \Psi(x,\M F) \, dx,
\end{align*}
which leads to
\begin{align*}
    |B_i| \leq \frac{2}{\lambda} \dashint_{B_i \cap \{ \Psi(x,Du) > \frac{\lambda}{4}\}} \Psi(x, \M Du) \, dx+ \frac{2}{\lambda} \dashint_{B_i \cap \{ \Psi(x,F) > \frac{\delta\lambda}{4}\}} \frac{1}{\delta}\Psi(x,\M F) \, dx.
\end{align*}
Substituting this into \eqref{PsiBi}, we find
\begin{align*}
    &\int_{5B_i \cap \{H(x,\M Du)>4c_l\lambda \}} \Psi(x,\M Du) \, dx  \\ &\leq 80^nS \int_{B_i \cap \{ \Psi(x,Du) > \frac{\lambda}{4}\}} \Psi(x, \M Du) \, dx+ 80^nS \int_{B_i \cap \{ \Psi(x,F) > \frac{\delta\lambda}{4}\}} \frac{1}{\delta}\Psi(x,\M F) \, dx.
\end{align*}
Since $\{5B_i\}$ is a covering of $E(r_1,\lambda)$ and $\{B_i\}$ is pairwise disjoint, we sum over $i$ to get the following level-set inequality
\begin{align} \label{PsiE}
    &\int_{ E(r_1,\lambda)} \Psi(x,\M Du) \, dx  \nonumber\\ &\leq 80^nS \int_{E(r_2,\frac{\lambda}{16c_l})} \Psi(x, \M Du) \, dx+ 80^nS \int_{B_{r_2} \cap \{ \Psi(x,F) > \frac{\delta\lambda}{16c_l}\}} \frac{1}{\delta}\Psi(x,\M F) \, dx,
\end{align}
for all $\lambda > \lambda_1=4c_l\lambda_0$.

\textit{Step 6 : Conclusion.}
We finalize the proof of Theorem \ref{Main} by employing a truncation argument and integrating over the level sets.
For $t>0$, we define the truncated potential as
\begin{align*}
    [\Psi(x,\M Du)]_t = \min \{ \Psi(x,\M Du),t\}.
\end{align*}
Let $\Ups \in \N$ be a Young function.
using Fubini's theorem and $\Ups'(0)=0$, we obtain for any $t > \lambda_1$,
\begin{align*}
    \int_{B_{r_1}} \Ups'([\Psi(x,\M Du)]_t) \Psi(x,\M Du) \, dx =  \int_0^{t} \Ups''(\lambda)\int_{E(r_1,\lambda)}  \Psi(x,\M Du) \, dxd\lambda \\
    \leq  \Ups'(\lambda_1)\int_{B_{r_2}} \Psi(x,\M Du)\, dx + \int_{\lambda_1}^{t} \Ups''(\lambda)\int_{E(r_1,\lambda)}  \Psi(x,\M Du) \, dxd\lambda.
\end{align*}
Note that we have
\begin{align*}
    \Ups'(\lambda_1)\int_{B_{r_2}} \Psi(x,\M Du)\, dx \leq \left( \frac{cR}{r_2-r_1}\right)^{n(1+s(\Ups))}\Ups\left( \dashint_{B_R} \Psi(x, \M Du) + \frac{1}{\delta} \Psi(x,\M F) \, dx\right).
\end{align*}
Next, we multiply the level-set inequality \eqref{PsiE} by $\Ups''(\lambda)$ and integrate over $\lambda \in (\lambda_1,t)$.
Then we get
\begin{align*}
    \int_{\lambda_1}^{t}& \Ups''(\lambda)\int_{E(r_1,\lambda)}  \Psi(x,\M Du) \, dxd\lambda \\
    &\leq 80^nS \int_{\lambda_1}^{t} \Ups''(\lambda)\int_{E(r_2,\frac{\lambda}{16c_l})} \Psi(x, \M Du) \, dxd\lambda \\
    &+ 80^nS \int_{0}^{\infty} \Ups''(\lambda)\int_{B_{r_2} \cap \{ \Psi(x,F) > \frac{\delta\lambda}{16c_l}\}} \frac{1}{\delta}\Psi(x,\M F) \, dx d\lambda .
\end{align*}
By applying Fubini's theorem, we obtain
\begin{align*}
    \int_{\lambda_1}^{t} \Ups''(\lambda)\int_{E(r_2,\frac{\lambda}{16c_l})} \Psi(x, \M Du) \, dxd\lambda \leq \til{c} c_l^{s(\Ups)} \int_{B_{r_2}} \Ups'([\Psi(x,\M Du)]_t)\Psi(x, \M Du) \, dx,
\end{align*}
\begin{align*}
    \int_{0}^{\infty} \Ups''(\lambda)\int_{B_{r_2} \cap \{ \Psi(x,F) > \frac{\delta\lambda}{16c_l}\}} \frac{1}{\delta}\Psi(x,\M F) \, dx d\lambda  &= \int_{B_{r_2}} \frac{1}{\delta}\Ups'\left(\frac{16c_l}{\delta} \Psi(x,\M F) \right) \Psi(x,\M F) \,dx\\
    &\leq \til{c} c_l^{s(\Ups)}\delta^{-(1+s(\Ups))} 
    \int_{B_{r_2}} \Ups(\Psi(x,\M F)) \, dx.
\end{align*}
Consequently, we arrive at
\begin{align*}
    \int_{B_{r_1}}& \Ups'([\Psi(x,\M Du)]_t) \Psi(x,\M Du) \, dx \ \\ &\leq \til{c} c_l^{s(\Ups)} S \int_{B_{r_2}} \Ups'([\Psi(x,\M Du)]_t)\Psi(x, \M Du) \, dx 
    + \til{c} c_l^{s(\Ups)} \delta^{-(1+s(\Ups))}S 
    \int_{B_{r_2}} \Ups(\Psi(x,\M F)) \, dx \\
    &+\Ups'(\lambda_1)\int_{B_{r_2}} \Psi(x,\M Du)\, dx.
\end{align*}
To apply Lemma \ref{tech}, we need to choose suitable $\epsilon$, $\delta$, $K$ and $r_0$ to ensure that
\begin{align*}
    \til{c} c_l^{s(\Ups)}S \leq \frac{1}{2}.
\end{align*}
First we select $K>1$ and $\epsilon\in (0,1)$ as
\begin{align*}
    K = \max\{ 8\til{c} c_l^{s(\Ups)}c_2,40\}, \quad \text{ and }\quad \epsilon = \frac{1}{8\til{c} c_l^{s(\Ups)}c_1}.
\end{align*}
Then, we choose sufficiently small $\delta_0 >0$ and $r_0 >0$ satisfying
\begin{align*}
    \delta \leq \min \left\{\left(\frac{1}{8\til{c} c_l^{s(\Ups)}c_\eps} \right)^{1/i_1},\delta_0\right\} \quad \text{ and }\quad r_0 \leq \left(\frac{1}{8\til{c} c_l^{s(\Ups)}c_K} \right)^{1/s_2}.
\end{align*}
With these choices, the inequality reduces to
\begin{align*}
    \int_{B_{r_1}}& \Ups'([\Psi(x,\M Du)]_t) \Psi(x,\M Du) \, dx \ \\ &\leq \frac{1}{2} \int_{B_{r_2}} \Ups'([\Psi(x,\M Du)]_t)\Psi(x, \M Du) \, dx 
    + c \int_{B_{r_2}} \Ups(\Psi(x,\M F)) \, dx \\
    &+\left(\frac{cR}{r_2-r_1}\right)^{n(1+s(\Ups))}\Ups\left( \dashint_{B_R} \Psi(x, \M Du) + c \Psi(x,\M F) \, dx\right),
\end{align*}
for any $R/2 \leq r_1<r_2<R$ and $t>\lambda_1$, where $c=c(\data, s(\Ups))>0$.
By applying Lemma \ref{tech}, we have
\begin{align*}
    \int_{B_{R/2}} \Ups'([\Psi(x,\M Du)]_t) \Psi(x,\M Du)  \, dx \ \leq c\Ups \left(\int_{B_{R}} \Psi(x, \M Du) \, dx \right)
    + c \int_{B_{R}} \Ups(\Psi(x,\M F)) \, dx .
\end{align*}
Finally, letting $t \rightarrow \infty$, we  obtain for any $R\leq r_0$
\begin{align*}
    \int_{B_{R/2}} \Ups(\Psi(x,\M Du))  \, dx \ \leq c\Ups \left(\int_{B_{R}} \Psi(x, \M Du) \, dx \right)
    + c \int_{B_{R}} \Ups(\Psi(x,\M F)) \, dx,
\end{align*}
 with $c=c(\data, s(\Ups))>0$.
 This completes the proof of Theorem \ref{Main}.
\end{proof}

\textbf{Data Availability}
Data sharing not applicable to this article as no datasets were generated or analyzed during the current study.

\textbf{Conflict of interest}
The authors declared that they have no conflict of interest to this work.

\bibliography{ref}

@article {Yang25,
    AUTHOR = {Byun, S.-S. and Yang, R.},
     TITLE = {Log-{BMO} matrix weights and quasilinear elliptic equations
              with {O}rlicz growth in {R}eifenberg domains},
   JOURNAL = {J. Lond. Math. Soc. (2)},
  FJOURNAL = {Journal of the London Mathematical Society. Second Series},
    VOLUME = {111},
      YEAR = {2025},
    NUMBER = {4},
     PAGES = {Paper No. e70151, 31},
      ISSN = {0024-6107,1469-7750},
   MRCLASS = {35B65 (35D30 35J25 35J62)},
  MRNUMBER = {4891024},
       DOI = {10.1112/jlms.70151},
       URL = {https://doi.org/10.1112/jlms.70151},
}

@article {Balci22,
    AUTHOR = {Balci, A. K. and Diening, L. and Giova, R. and
              Passarelli di Napoli, A.},
     TITLE = {Elliptic equations with degenerate weights},
   JOURNAL = {SIAM J. Math. Anal.},
  FJOURNAL = {SIAM Journal on Mathematical Analysis},
    VOLUME = {54},
      YEAR = {2022},
    NUMBER = {2},
     PAGES = {2373--2412},
      ISSN = {0036-1410,1095-7154},
  MRNUMBER = {4410267},
   MRCLASS = {35B65 (35J70 35R05)},
MRREVIEWER = {Xiumin\ Du},
       DOI = {10.1137/21M1412529},
       URL = {https://doi.org/10.1137/21M1412529},
}

@article {Byun20,
    AUTHOR = {Byun, S.-S. and Oh, J.},
     TITLE = {Regularity results for generalized double phase functionals},
   JOURNAL = {Anal. PDE},
  FJOURNAL = {Analysis \& PDE},
    VOLUME = {13},
      YEAR = {2020},
    NUMBER = {5},
     PAGES = {1269--1300},
      ISSN = {2157-5045,1948-206X},
   MRCLASS = {49N60 (35B65 35J20 49J10)},
  MRNUMBER = {4149062},
MRREVIEWER = {Elvira\ Mascolo},
       DOI = {10.2140/apde.2020.13.1269},
       URL = {https://doi.org/10.2140/apde.2020.13.1269},
}

@article {Baasandorj20,
    AUTHOR = {Baasandorj, S. and Byun, S.-S. and Oh, J.},
     TITLE = {Calder\'on-{Z}ygmund estimates for generalized double phase
              problems},
   JOURNAL = {J. Funct. Anal.},
  FJOURNAL = {Journal of Functional Analysis},
    VOLUME = {279},
      YEAR = {2020},
    NUMBER = {7},
     PAGES = {108670, 57},
      ISSN = {0022-1236,1096-0783},
   MRCLASS = {35J70 (35B65 42B25 46E35)},
  MRNUMBER = {4107816},
MRREVIEWER = {Juha\ K.\ Kinnunen},
       DOI = {10.1016/j.jfa.2020.108670},
       URL = {https://doi.org/10.1016/j.jfa.2020.108670},
}

@article {Zhikov95,
    AUTHOR = {Zhikov, V. V.},
     TITLE = {On {L}avrentiev's phenomenon},
   JOURNAL = {Russian J. Math. Phys.},
  FJOURNAL = {Russian Journal of Mathematical Physics},
    VOLUME = {3},
      YEAR = {1995},
    NUMBER = {2},
     PAGES = {249--269},
      ISSN = {1061-9208},
  MRNUMBER = {1350506},
   MRCLASS = {49J45 (49J10)},
MRREVIEWER = {Philip\ D.\ Loewen},
}

@article {Zhikov86,
    AUTHOR = {Zhikov, V. V.},
     TITLE = {Averaging of functionals of the calculus of variations and
              elasticity theory},
   JOURNAL = {Izv. Akad. Nauk SSSR Ser. Mat.},
  FJOURNAL = {Izvestiya Akademii Nauk SSSR. Seriya Matematicheskaya},
    VOLUME = {50},
      YEAR = {1986},
    NUMBER = {4},
     PAGES = {675--710, 877},
      ISSN = {0373-2436},
   MRCLASS = {49H05 (73C60)},
  MRNUMBER = {864171},
MRREVIEWER = {Vadim\ Komkov},
}

@article {Colombo15,
    AUTHOR = {Colombo, M. and Mingione, G.},
     TITLE = {Regularity for double phase variational problems},
   JOURNAL = {Arch. Ration. Mech. Anal.},
  FJOURNAL = {Archive for Rational Mechanics and Analysis},
    VOLUME = {215},
      YEAR = {2015},
    NUMBER = {2},
     PAGES = {443--496},
      ISSN = {0003-9527,1432-0673},
   MRCLASS = {49N60 (35B27 35B65)},
  MRNUMBER = {3294408},
MRREVIEWER = {Eugen\ Viszus},
       DOI = {10.1007/s00205-014-0785-2},
       URL = {https://doi.org/10.1007/s00205-014-0785-2},
}

@article {Colombo152,
    AUTHOR = {Colombo, M. and Mingione, G.},
     TITLE = {Bounded minimisers of double phase variational integrals},
   JOURNAL = {Arch. Ration. Mech. Anal.},
  FJOURNAL = {Archive for Rational Mechanics and Analysis},
    VOLUME = {218},
      YEAR = {2015},
    NUMBER = {1},
     PAGES = {219--273},
      ISSN = {0003-9527,1432-0673},
   MRCLASS = {49J10 (49N60)},
  MRNUMBER = {3360738},
MRREVIEWER = {Helmut\ Kaul},
       DOI = {10.1007/s00205-015-0859-9},
       URL = {https://doi.org/10.1007/s00205-015-0859-9},
}

@article {Baasandorj25,
    AUTHOR = {Baasandorj, S. and Byun, S.-S.},
     TITLE = {Regularity for {O}rlicz {P}hase {P}roblems},
   JOURNAL = {Mem. Amer. Math. Soc.},
  FJOURNAL = {Memoirs of the American Mathematical Society},
    VOLUME = {308},
      YEAR = {2025},
    NUMBER = {1556},
     PAGES = {iii+131},
      ISSN = {0065-9266,1947-6221},
      ISBN = {978-1-4704-7304-4; 978-1-4704-8143-8},
   MRCLASS = {49N60 (35B27 35B65 35J20 49J10)},
  MRNUMBER = {4896165},
       DOI = {10.1090/memo/1556},
       URL = {https://doi.org/10.1090/memo/1556},
}

@article {Cho25,
    AUTHOR = {Byun, S.-S. and Cho, Y. and Ryu, S.},
    title={Calder\'{o}n-{Z}ygmund estimates for double phase problems with matrix weights}, 
   JOURNAL = {Calc. Var. Partial Differential Equations},
  FJOURNAL = {Calculus of Variations and Partial Differential Equations},
    VOLUME = {65},
      YEAR = {2026},
    NUMBER = {4},
     PAGES = {Paper No. 121},
      ISSN = {0944-2669,1432-0835},
   MRCLASS = {35B65 (35J70 35J75)},
  MRNUMBER = {5041108},
       DOI = {10.1007/s00526-026-03285-6},
       URL = {https://doi.org/10.1007/s00526-026-03285-6},
}

@book {Kokilashvili91,
    AUTHOR = {Kokilashvili, V. and Krbec, M.},
     TITLE = {Weighted inequalities in {L}orentz and {O}rlicz spaces},
 PUBLISHER = {World Scientific Publishing Co., Inc., River Edge, NJ},
      YEAR = {1991},
     PAGES = {xii+233},
      ISBN = {981-02-0612-7},
   MRCLASS = {42B25 (26D10 42B10 46E30 47B38 47G10)},
  MRNUMBER = {1156767},
MRREVIEWER = {Vladimir\ D.\ Stepanov},
       DOI = {10.1142/9789814360302},
       URL = {https://doi.org/10.1142/9789814360302},
}

@article {Balci23,
    AUTHOR = {Balci, A. K. and Byun, S.-S. and Diening, L. and Lee,
              H.-S.},
     TITLE = {Global maximal regularity for equations with degenerate
              weights},
   JOURNAL = {J. Math. Pures Appl. (9)},
  FJOURNAL = {Journal de Math\'ematiques Pures et Appliqu\'ees. Neuvi\`eme
              S\'erie},
    VOLUME = {177},
      YEAR = {2023},
     PAGES = {484--530},
      ISSN = {0021-7824,1776-3371},
   MRCLASS = {35J70 (35B65 35J25)},
  MRNUMBER = {4629762},
MRREVIEWER = {Filomena\ Feo},
       DOI = {10.1016/j.matpur.2023.07.010},
       URL = {https://doi.org/10.1016/j.matpur.2023.07.010},
}

@article {Cho24,
    AUTHOR = {Byun, S.-S. and Cho, Y.},
     TITLE = {Calder\'on-{Z}ygmund estimates for nonlinear elliptic obstacle
              problems with log-{BMO} matrix weights},
   JOURNAL = {Discrete Contin. Dyn. Syst. Ser. B},
  FJOURNAL = {Discrete and Continuous Dynamical Systems. Series B. A Journal
              Bridging Mathematics and Sciences},
    VOLUME = {29},
      YEAR = {2024},
    NUMBER = {12},
     PAGES = {4772--4792},
      ISSN = {1531-3492,1553-524X},
  MRNUMBER = {4795497},
   MRCLASS = {35J87 (35B65 35J70 35J75 49J40)},
       DOI = {10.3934/dcdsb.2024065},
       URL = {https://doi.org/10.3934/dcdsb.2024065},
}

@article {Cho242,
    AUTHOR = {Byun, S.-S. and Cho, Y. and Lee, H.-S.},
     TITLE = {Elliptic equations with matrix weights and measurable
              nonlinearities on nonsmooth domains},
   JOURNAL = {Proc. Amer. Math. Soc.},
  FJOURNAL = {Proceedings of the American Mathematical Society},
    VOLUME = {152},
      YEAR = {2024},
    NUMBER = {7},
     PAGES = {2963--2982},
      ISSN = {0002-9939,1088-6826},
   MRCLASS = {35B65 (35J70 35J75)},
  MRNUMBER = {4753281},
MRREVIEWER = {Haitao\ Wan},
       DOI = {10.1090/proc/16770},
       URL = {https://doi.org/10.1090/proc/16770},
}

@article{Byun25,
  AUTHOR  = {Byun, S.-S. and Jung, H. and Kim, H.},
  TITLE   = {Absence of the Lavrentiev Phenomenon for Generalized Double-Phase Functionals with Variable Matrix Weights},
  JOURNAL = {Proc. Amer. Math. Soc.},
  YEAR    = {2026},
  NOTE    = { \url{https://doi.org/10.1090/proc/17602}},
}

@article {Baroni18,
    AUTHOR = {Baroni, P. and Colombo, M. and Mingione, G.},
     TITLE = {Regularity for general functionals with double phase},
   JOURNAL = {Calc. Var. Partial Differential Equations},
  FJOURNAL = {Calculus of Variations and Partial Differential Equations},
    VOLUME = {57},
      YEAR = {2018},
    NUMBER = {2},
     PAGES = {Paper No. 62, 48},
      ISSN = {0944-2669,1432-0835},
   MRCLASS = {49N60},
MRREVIEWER = {Elvira\ Mascolo},
       DOI = {10.1007/s00526-018-1332-z},
       URL = {https://doi.org/10.1007/s00526-018-1332-z},
}

@article {Colombo16,
    AUTHOR = {Colombo, M. and Mingione, G.},
     TITLE = {Calder\'on-{Z}ygmund estimates and non-uniformly elliptic
              operators},
   JOURNAL = {J. Funct. Anal.},
  FJOURNAL = {Journal of Functional Analysis},
    VOLUME = {270},
      YEAR = {2016},
    NUMBER = {4},
     PAGES = {1416--1478},
      ISSN = {0022-1236,1096-0783},
   MRCLASS = {35J62 (35B65 35J70)},
  MRNUMBER = {3447716},
MRREVIEWER = {Francesco\ Della Pietra},
       DOI = {10.1016/j.jfa.2015.06.022},
       URL = {https://doi.org/10.1016/j.jfa.2015.06.022},
}

@article {Lieberman91,
    AUTHOR = {Lieberman, G. M.},
     TITLE = {The natural generalization of the natural conditions of
              {L}adyzhenskaya and {U}ralcprime tseva for elliptic
              equations},
   JOURNAL = {Comm. Partial Differential Equations},
  FJOURNAL = {Communications in Partial Differential Equations},
    VOLUME = {16},
      YEAR = {1991},
    NUMBER = {2-3},
     PAGES = {311--361},
      ISSN = {0360-5302,1532-4133},
   MRCLASS = {35J60 (35B65)},
  MRNUMBER = {1104103},
MRREVIEWER = {M.\ Biroli},
       DOI = {10.1080/03605309108820761},
       URL = {https://doi.org/10.1080/03605309108820761},
}

@book {Han97,
    AUTHOR = {Han, Q. and Lin, F.},
     TITLE = {Elliptic partial differential equations},
    SERIES = {Courant Lecture Notes in Mathematics},
    VOLUME = {1},
 PUBLISHER = {New York University, Courant Institute of Mathematical
              Sciences, New York; American Mathematical Society, Providence,
              RI},
      YEAR = {1997},
     PAGES = {x+144},
      ISBN = {0-9658703-0-8; 0-8218-2691-3},
   MRCLASS = {35Jxx (35-01 35B50)},
  MRNUMBER = {1669352},
}

@article {Mingione07,
    AUTHOR = {Acerbi, E. and Mingione, G.},
     TITLE = {Gradient estimates for a class of parabolic systems},
   JOURNAL = {Duke Math. J.},
  FJOURNAL = {Duke Mathematical Journal},
    VOLUME = {136},
      YEAR = {2007},
    NUMBER = {2},
     PAGES = {285--320},
      ISSN = {0012-7094,1547-7398},
   MRCLASS = {35K50 (35B45 35K55 35K65)},
  MRNUMBER = {2286632},
MRREVIEWER = {Siegfried\ Carl},
       DOI = {10.1215/S0012-7094-07-13623-8},
       URL = {https://doi.org/10.1215/S0012-7094-07-13623-8},
}

@book {Harjulehto19,
    AUTHOR = {Harjulehto, P. and H\"ast\"o, P.},
     TITLE = {Orlicz spaces and generalized {O}rlicz spaces},
    SERIES = {Lecture Notes in Mathematics},
    VOLUME = {2236},
 PUBLISHER = {Springer, Cham},
      YEAR = {2019},
     PAGES = {x+167},
      ISBN = {978-3-030-15099-0; 978-3-030-15100-3},
   MRCLASS = {46-02 (42B35 46E35)},
  MRNUMBER = {3931352},
MRREVIEWER = {Karol\ Le\'snik},
       DOI = {10.1007/978-3-030-15100-3},
       URL = {https://doi.org/10.1007/978-3-030-15100-3},
}

@book {Musielak83,
    AUTHOR = {Musielak, J.},
     TITLE = {Orlicz spaces and modular spaces},
    SERIES = {Lecture Notes in Mathematics},
    VOLUME = {1034},
 PUBLISHER = {Springer-Verlag, Berlin},
      YEAR = {1983},
     PAGES = {iii+222},
      ISBN = {3-540-12706-2},
   MRCLASS = {46E30 (46-02 46A15)},
  MRNUMBER = {724434},
MRREVIEWER = {K.\ Sundaresan},
       DOI = {10.1007/BFb0072210},
       URL = {https://doi.org/10.1007/BFb0072210},
}

@article {Filippis19,
    AUTHOR = {De Filippis, C. and Mingione, G.},
     TITLE = {A borderline case of {C}alder\'on-{Z}ygmund estimates for
              nonuniformly elliptic problems},
   JOURNAL = {Algebra i Analiz},
  FJOURNAL = {Rossi\u iskaya Akademiya Nauk. Algebra i Analiz},
    VOLUME = {31},
      YEAR = {2019},
    NUMBER = {3},
     PAGES = {82--115},
      ISSN = {0234-0852},
   MRCLASS = {35J70 (35B27 35B45 35J62 35J92 49J10 49N60)},
  MRNUMBER = {3985927},
       DOI = {10.1090/spmj/1608},
       URL = {https://doi.org/10.1090/spmj/1608},
}

@article {Benkirane14,
    AUTHOR = {Benkirane, A. and Sidi El Vally, M.},
     TITLE = {Variational inequalities in {M}usielak-{O}rlicz-{S}obolev
              spaces},
   JOURNAL = {Bull. Belg. Math. Soc. Simon Stevin},
  FJOURNAL = {Bulletin of the Belgian Mathematical Society. Simon Stevin},
    VOLUME = {21},
      YEAR = {2014},
    NUMBER = {5},
     PAGES = {787--811},
      ISSN = {1370-1444,2034-1970},
   MRCLASS = {35J40 (35J62 46A80 46E35)},
  MRNUMBER = {3298478},
       URL = {http://projecteuclid.org/euclid.bbms/1420071854},
}

@article {Harjulehto16,
    AUTHOR = {Harjulehto, P. and H\"ast\"o, P. and Kl\'en, R.},
     TITLE = {Generalized {O}rlicz spaces and related {PDE}},
   JOURNAL = {Nonlinear Anal.},
  FJOURNAL = {Nonlinear Analysis. Theory, Methods \& Applications. An
              International Multidisciplinary Journal},
    VOLUME = {143},
      YEAR = {2016},
     PAGES = {155--173},
      ISSN = {0362-546X,1873-5215},
   MRCLASS = {35J60 (35J20 46E30 46E35 49J40)},
  MRNUMBER = {3516828},
MRREVIEWER = {Futoshi\ Takahashi},
       DOI = {10.1016/j.na.2016.05.002},
       URL = {https://doi.org/10.1016/j.na.2016.05.002},
}
\end{document}